\documentclass{amsart}
\usepackage{amsmath,amssymb,amsthm}
\usepackage{graphicx}
\usepackage{subcaption}
\usepackage{algorithm}
\usepackage{algpseudocode}
\usepackage{placeins}
\usepackage[margin=1.0in]{geometry}
\usepackage[hidelinks]{hyperref}

\theoremstyle{plain}
\newtheorem{thm}{Theorem}[section]
\newtheorem{prop}[thm]{Proposition}
\newtheorem{lemma}[thm]{Lemma}
\newtheorem{cor}[thm]{Corollary}

\theoremstyle{definition}
\newtheorem{defn}[thm]{Definition}
\newtheorem{exmp}[thm]{Example}
\newtheorem{rem}[thm]{Remark}

\newcommand{\A}{\mathbb{C}^*\times\mathbb{C}}
\newcommand{\B}{B_{\frac{1}{2}}\left(\frac{1}{2},0,0\right)}
\newcommand{\cp}{\mathbb{P}^2_\mathbb{C}}
\DeclareMathOperator{\Real}{Re}
\DeclareMathOperator{\Imag}{Im}
\DeclareMathOperator{\characteristic}{char}
\DeclareMathOperator{\interior}{int}
\numberwithin{equation}{section}
\numberwithin{algorithm}{section}
\numberwithin{figure}{section}

\title{Visualization of Complex Projective Curves}
\author{Seth Dutter}
\address{Department of Mathematics, Statistics and Computer Science, University of Wisconsin--Stout Polytechnic}
\email{dutters@uwstout.edu}

\date{August 4, 2026}

\hypersetup{pdftitle={Visualization of Complex Projective Curves}, pdfauthor={Seth Dutter}}


\begin{document}

\begin{abstract}
We introduce a nonlinear map $\alpha:\mathbb{C}^2\rightarrow\mathbb{R}^3$ with the purpose of visualizing curves. Basic properties of $\alpha$ are proved, including preservation of orthogonality, recovery of the magnitudes of vectors in the preimage, and continuous extension of $\alpha$ to  $\widetilde{\alpha}:\mathbb{P}^2_\mathbb{C}\rightarrow\mathbb{R}^3$. For plane curves $Z\subset\mathbb{P}^2_\mathbb{C}$, it is proved that $\widetilde{\alpha}(Z)$ is the union of boundaries of star-shaped domains. Methods are established to descend finite-order automorphisms of smooth projective curves to rotations of their images in $\mathbb{R}^3$. Efficient techniques for creating meshes and ray-traced images of $\widetilde{\alpha}(Z)$ are developed.
\end{abstract}

\maketitle

\section{Introduction}

Let $X\subset \mathbb{C}^2$ denote an irreducible algebraic curve defined by a complex polynomial $f(u,v)=0$. Writing $u=u_0+iu_1$ and $v=v_0+iv_1$, we identify $\mathbb{C}^2$ with $\mathbb{R}^4$ by
\[
(u,v)\mapsto (u_0,u_1,v_0,v_1).
\]
Under this identification, the smooth locus of $X$ is a real two-dimensional manifold in four-dimensional space. Although its genus is readily computable by algebraic means, this alone does not determine what the surface looks like as a subset of $\mathbb{R}^4$.

In order to visualize such a surface on a computer, one typically projects from $\mathbb{R}^4$ to $\mathbb{R}^3$. A standard choice is to discard a coordinate, for example, by sending
\[
(u_0,u_1,v_0,v_1)\mapsto (u_0,u_1,v_0).
\]
In the case of a complex curve, we can write its defining equation, $f(u,v)=0$, as a system of two real equations and algebraically eliminate a variable. Unfortunately, for a general complex polynomial of degree $d$, this process results in an implicitly defined surface in $\mathbb{R}^3$ of degree $d^2$. The number of monomials in three variables of degree at most $d^2$ is $\binom{d^2+3}{3}\sim \frac{1}{6}d^6$, which can become impractical to work with for even moderately sized $d$.

To compound matters, ray tracing such a surface requires computing normal vectors, which in turn necessitates either explicit partial derivatives of the resultant or numerical approximation of these same partial derivatives. This approach is far more computationally expensive than working directly with the original polynomial in two complex variables. In the special case that the curve is rational, one can instead parameterize the curve and project the resulting mesh. For more general curves, algorithms exist which track branches and assemble a mesh in $\mathbb{C}^2$ before projecting to $\mathbb{R}^3$. See, for example, \cite{Kranich2015}. These approaches offer substantial improvements in speed.

Unfortunately, regardless of how it is computed, such a linear projection loses most of the structure of $\mathbb{C}^2$ and $X$. The norm of a vector in $\mathbb{C}^2$ is no longer recoverable, a continuous extension to $\cp$ is not possible, and information about orthogonality is lost. In this paper we develop a function which addresses the above issues while also allowing for computationally efficient rendering of the resulting real surfaces and visualization of finite-order automorphisms.

In Section \ref{section:projection} we define and then prove basic properties of the function with a focus on the image of $\mathbb{C}^2$. We next study the geometry of the image of curves in Section \ref{section:geometry}, including an exploration of how automorphisms of projective curves descend to rotations of their images. Finally, in Section \ref{section:visualization}, efficient algorithms for mesh construction and ray tracing are established.

\begin{figure}[!ht]
    \centering
    \begin{subfigure}{0.3\textwidth}
        \centering
        \includegraphics[height=1.5in]{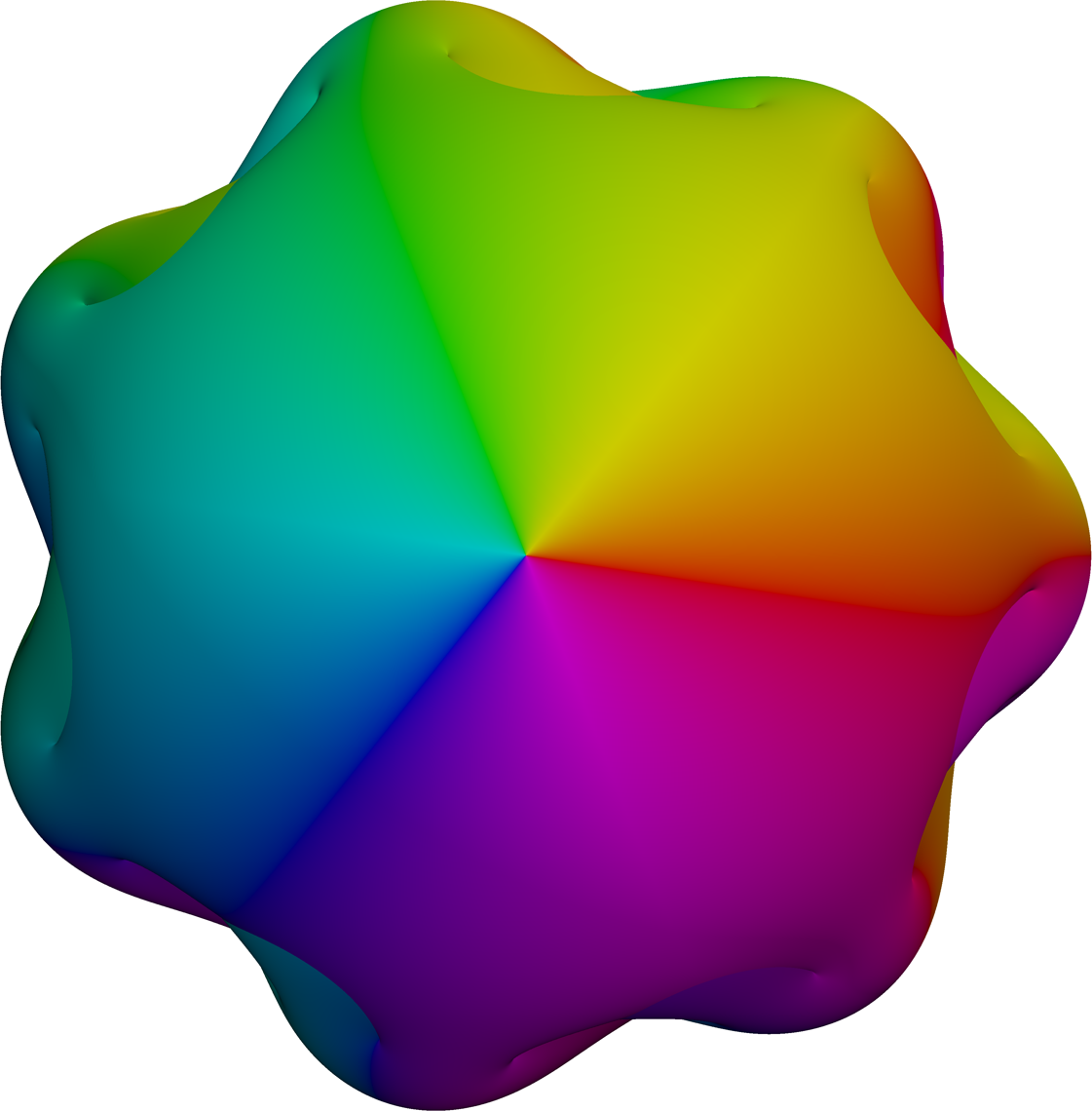}
        \caption{Rear view}
    \end{subfigure}
    \hfill
    \begin{subfigure}{0.3\textwidth}
        \centering
        \includegraphics[height=1.5in]{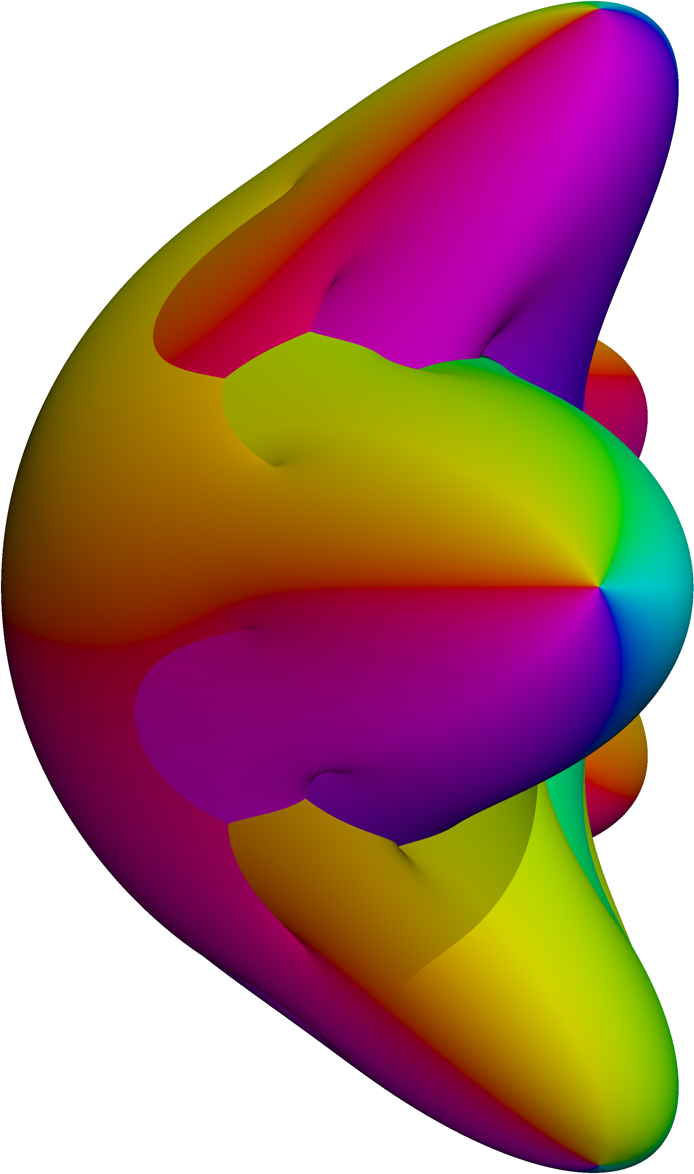}
        \caption{Side view}
    \end{subfigure}
    \hfill
    \begin{subfigure}{0.3\textwidth}
        \centering
        \includegraphics[height=1.5in]{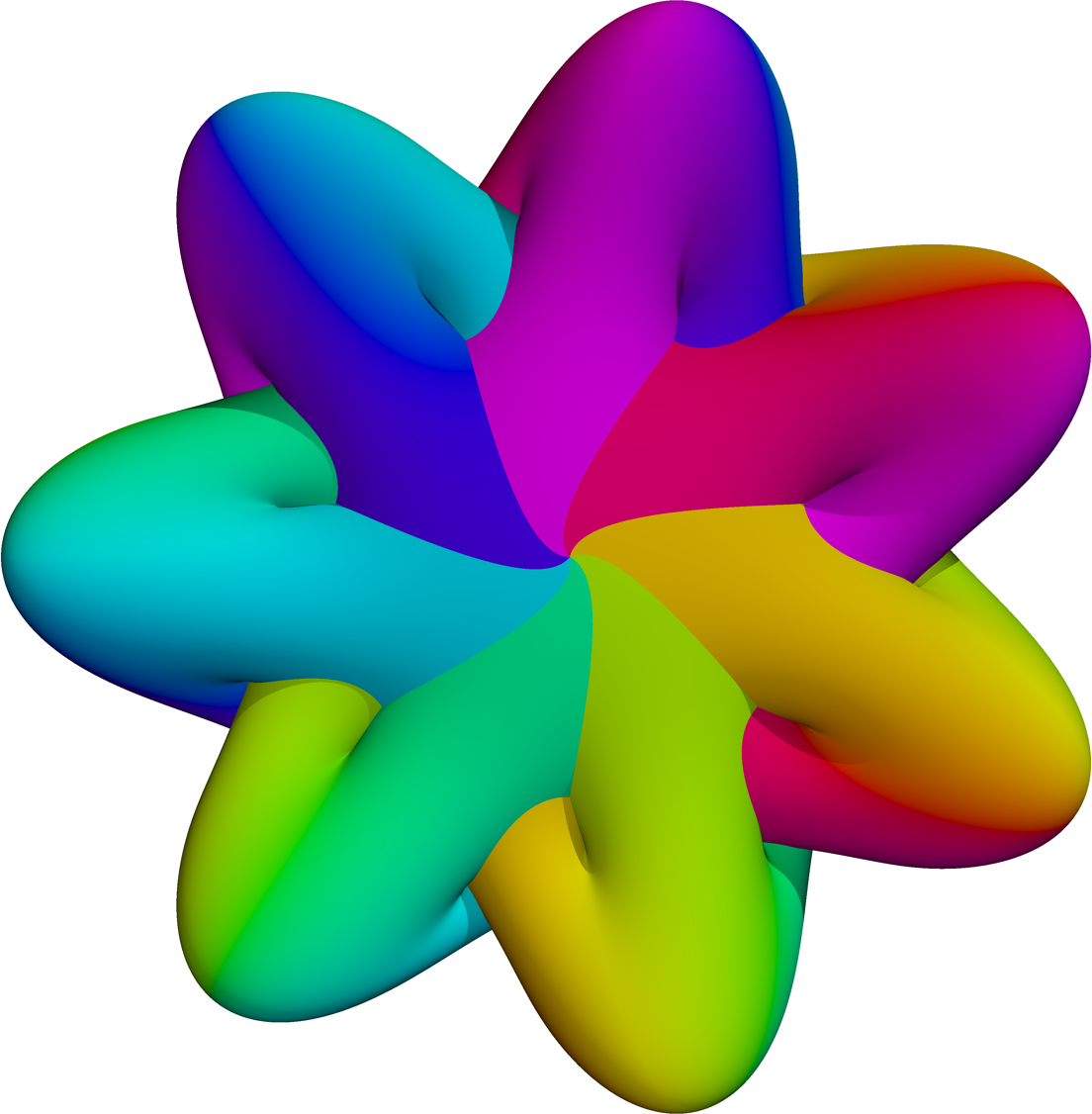}
        \caption{Front view}
    \end{subfigure}
    \caption{Ray-traced views of the genus $15$ projective curve $\left(v+\frac{i}{2}w\right)^7+iu^7+w^7=0$. The automorphism $[u:v:w]\mapsto[e^{2\pi i/7}u:v:w]$ can be seen in the rotational symmetry of the image.}
    \label{fig:intro}
\end{figure}

\section{The projection \texorpdfstring{$\alpha$}{alpha}}\label{section:projection}

We write $(x,y,z)$ for coordinates on $\mathbb{R}^3$ and identify $\mathbb{R}^3$ with $\mathbb{R}\times \mathbb{C}$ by $(x,y,z)\mapsto (x,y+iz)$ whenever it is notationally convenient. Additionally, we define
\[
B=\B\subset \mathbb{R}^3
\]
to be the open ball of radius $\frac{1}{2}$ centered at $(\frac{1}{2},0,0)$. The map studied throughout the paper, $\alpha:\mathbb{C}^2\rightarrow\mathbb{R}^3$, is defined by
\[
\alpha(u,v)=\left(\frac{|u|^2}{1+|u|^2+|v|^2},\frac{\bar{u}v}{1+|u|^2+|v|^2}\right).
\]
Because $\alpha(0,v)=(0,0,0)$ for every $v\in\mathbb{C}$, we will focus on the restricted domain $\A$. Under this restriction, the magnitude of a point in $\mathbb{C}^2$ is calculable from  its image (Lemma \ref{lem:magnitude}), the range is exactly $B$ (Proposition \ref{prop:image_ball}), $\alpha$ only identifies points if they differ by multiplication by a scalar in $S^1$ (Corollary \ref{cor:circle_action}), and orthogonality of vectors in $\mathbb{C}^2$ is preserved (Theorem \ref{thm:orthogonality}). Finally, we show that $\alpha$ extends continuously to $\cp$ (Theorem \ref{thm:projective_extension}).

\begin{lemma}\label{lem:magnitude}
Let $(u,v)\in\A$ and let $\alpha(u,v)=(x,y,z)$. Then $\alpha(\A)\subset B$ and
\[
|u|^2+|v|^2=\frac{x^2+y^2+z^2}{x-x^2-y^2-z^2}.
\]
\end{lemma}

\begin{proof}
By the definition of $\alpha$,
\[
x=\frac{|u|^2}{1+|u|^2+|v|^2}, \qquad y^2+z^2 = \frac{|uv|^2}{(1+|u|^2+|v|^2)^2}.
\]
Therefore,
\[
x^2+y^2+z^2=\frac{|u|^4}{(1+|u|^2+|v|^2)^2} + \frac{|uv|^2}{(1+|u|^2+|v|^2)^2}.
\]
The two summands on the right-hand side share a common factor of $|u|^2/(1+|u|^2+|v|^2)$. Pulling this factor out yields
\begin{align*}
x^2+y^2+z^2 &= \left(\frac{|u|^2}{1+|u|^2+|v|^2}\right)\left(\frac{|u|^2+|v|^2}{1+|u|^2+|v|^2}\right)\\
&= x \frac{|u|^2+|v|^2}{1+|u|^2+|v|^2}.
\end{align*}
The second equality comes from substituting back for $x$. After clearing denominators, we can manipulate the above equation to get
\begin{equation}\label{eq:norm}
x^2+y^2+z^2 = (x-x^2-y^2-z^2)(|u|^2+|v|^2).
\end{equation}
Since $u\neq 0$ and hence $x\neq 0$, it follows that $x-x^2-y^2-z^2 > 0$, which is a defining inequality for $B$. Dividing both sides of Equation \eqref{eq:norm} by $x-x^2-y^2-z^2$ gives the claimed formula for $|u|^2+|v|^2$.
\end{proof}

The hypothesis that $u\neq 0$ is essential. If $u=0$, then $\alpha(0,v)=(0,0,0)$ for every $v\in\mathbb{C}$, so the norm of $(0,v)$ cannot be recovered from the image.

\begin{thm}\label{thm:uv_formula}
Let $(u,v)\in\A$ and $\alpha(u,v)=(x,y,z)$. Then there exists $\lambda\in S^1$ such that
\[
(u,v)=\left(\frac{x\lambda}{\sqrt{x-x^2-y^2-z^2}},\frac{(y+iz)\lambda}{\sqrt{x-x^2-y^2-z^2}}\right).
\]
\end{thm}

\begin{proof}
Without further mention, we will use the fact that $|u|\neq 0$ and therefore $x>0$. By the definition of $\alpha$, $|u|^2/(1+|u|^2+|v|^2) = x$. Clearing denominators and applying Lemma \ref{lem:magnitude} gives us
\begin{align*}
|u|^2 &= x(1+|u|^2+|v|^2)\\
&= x\left(1 + \frac{x^2+y^2+z^2}{x-x^2-y^2-z^2}\right)\\
&= \frac{x^2}{x-x^2-y^2-z^2}.
\end{align*}
This, after taking square roots, becomes
\[
|u| = \frac{|x|}{\sqrt{x-x^2-y^2-z^2}} = \frac{x}{\sqrt{x-x^2-y^2-z^2}}.
\]
Therefore, there exists some $\lambda \in S^1$ such that
\begin{equation}\label{eq:u_formula}
u = \frac{x\lambda}{\sqrt{x-x^2-y^2-z^2}}.
\end{equation}
For the $v$-coordinate, we observe that, by the definition of $\alpha$,
\[
\frac{y+iz}{x} = \left(\frac{\bar{u}v}{1+|u|^2+|v|^2}\right)\bigg/\left(\frac{|u|^2}{1+|u|^2+|v|^2}\right) = \frac{v}{u}.
\]
Multiplying through by $u$ and then substituting in Equation \eqref{eq:u_formula} gives
\begin{align*}
v &= \frac{y+iz}{x} u\\
&= \frac{y+iz}{x}\frac{x\lambda}{\sqrt{x-x^2-y^2-z^2}}\\
&= \frac{(y+iz)\lambda}{\sqrt{x-x^2-y^2-z^2}},
\end{align*}
as claimed.
\end{proof}

\begin{prop}\label{prop:image_ball}
The restriction $\alpha: \A\to B$ is surjective. In particular, $\alpha(\A)=B$.
\end{prop}

\begin{proof}
By Lemma \ref{lem:magnitude}, $\alpha(\A)\subset B$. Conversely, let $(x,y,z)\in B$ and define
\[
u=\frac{x}{\sqrt{x-x^2-y^2-z^2}},
\qquad
v=\frac{y+iz}{\sqrt{x-x^2-y^2-z^2}}.
\]
Then
\[
1+|u|^2+|v|^2=1+\frac{x^2+y^2+z^2}{x-x^2-y^2-z^2}=\frac{x}{x-x^2-y^2-z^2}.
\]
Therefore
\[
\frac{|u|^2}{1+|u|^2+|v|^2}=\frac{x^2/(x-x^2-y^2-z^2)}{x/(x-x^2-y^2-z^2)}=x
\]
and
\[
\frac{\bar{u}v}{1+|u|^2+|v|^2}=\frac{x(y+iz)/(x-x^2-y^2-z^2)}{x/(x-x^2-y^2-z^2)}=y+iz.
\]
Thus $\alpha(u,v)=(x,y,z)$, so every point of $B$ lies in the image.
\end{proof}

\begin{cor}\label{cor:circle_action}
Let $(u,v),(u',v')\in\A$. Then $\alpha(u,v)=\alpha(u',v')$ if and only if there exists $\lambda\in S^1$ such that
\[
(u',v')=(\lambda u,\lambda v).
\]
\end{cor}

\begin{proof}
If $(u',v')=(\lambda u,\lambda v)$ with $\lambda\in S^1$, then the fact that $\alpha(u',v')=\alpha(u,v)$ follows immediately from the definition of $\alpha$. Conversely, suppose $\alpha(u,v)=\alpha(u',v')=(x,y,z)$. By Theorem \ref{thm:uv_formula}, there exist $\mu, \mu'\in S^1$ such that
\[
(u,v)=\left(\frac{x\mu}{\sqrt{x-x^2-y^2-z^2}},\frac{(y+iz)\mu}{\sqrt{x-x^2-y^2-z^2}}\right)
\]
and
\[
(u',v')=\left(\frac{x\mu'}{\sqrt{x-x^2-y^2-z^2}},\frac{(y+iz)\mu'}{\sqrt{x-x^2-y^2-z^2}}\right).
\]
Hence $(u',v')/\mu'=(u,v)/\mu$. Letting $\lambda = \mu'/\mu\in S^1$ completes the proof.
\end{proof}

\begin{thm}\label{thm:orthogonality}
Let $(u,v),(u',v')\in\A$ and
\[
\alpha(u,v)=(x,y,z),\qquad \alpha(u',v')=(x',y',z').
\]
Then $\langle (u,v),(u',v')\rangle=0$ if and only if $(x,y,z)\cdot (x',y',z')=0$ and the vectors $(y,z)$ and $(y',z')$ are linearly dependent over $\mathbb{R}$.
\end{thm}

\begin{proof}
By Theorem \ref{thm:uv_formula}, there exist nonzero complex numbers $a, a'\in\mathbb{C}^*$ such that
\[
(u,v)=a(x,y+iz), \qquad (u',v')=a'(x',y'+iz').
\]
Using the convention that the Hermitian inner product is linear in the first variable and conjugate linear in the second, we obtain
\[
\langle (u,v),(u',v')\rangle=a\overline{a'}\left(xx'+(y+iz)(y'-iz')\right).
\]
The scalar in front is nonzero, so the inner product vanishes if and only if
\[
xx'+(y+iz)(y'-iz')=0.
\]
Expanding the second term gives
\[
xx'+yy'+zz'+i(y'z-yz')=0.
\]
Thus the Hermitian inner product is zero if and only if
\[
(x,y,z)\cdot(x',y',z')=0
\]
and
\[
y'z-yz'=0.
\]
The last condition is equivalent to the linear dependence of $(y,z)$ and $(y',z')$ over $\mathbb{R}$.
\end{proof}

As a consequence, orthogonality in $\mathbb{C}^2$ forces orthogonality of the images in $\mathbb{R}^3$. This implication continues to hold when $u=0$ or $u'=0$, since such vectors map to $(0,0,0)$.

\begin{thm}\label{thm:projective_extension}
The function $\alpha: \mathbb{C}^2\to \mathbb{R}\times\mathbb{C}$ extends continuously to a function $\widetilde{\alpha}: \cp\rightarrow \overline{B}$ given by
\[
\widetilde{\alpha}([u:v:w])=\left(\frac{|u|^2}{|u|^2+|v|^2+|w|^2}, \frac{\overline{u}v}{|u|^2+|v|^2+|w|^2}\right).
\]
Moreover, the line at infinity, $L_\infty=\{[u:v:w]\in\cp:w=0\}$, maps onto $\partial B$.
\end{thm}

\begin{proof}
The formula for $\widetilde{\alpha}$ is well defined on projective classes because multiplying $(u,v,w)$ by some $\xi\in\mathbb{C}^*$ multiplies each numerator and the denominator in the formula for $\widetilde{\alpha}$ by the same factor, $|\xi|^2$. Continuity is immediate, and on the affine chart $w=1$ the formula reduces to the original definition of $\alpha$.

Write $\widetilde{\alpha}([u:v:w])=(x,y,z)$. Then a direct calculation, left to the reader, gives us
\[
x-x^2-y^2-z^2=\frac{|u|^2|w|^2}{\left(|u|^2+|v|^2+|w|^2\right)^2}\ge 0.
\]
Therefore, the image lies in $\overline{B}$. If $w=0$, then equality holds and the image lies on $\partial B$.

Conversely, let $(x,y,z)\in\partial B$, so $x^2+y^2+z^2=x$. If $x=0$, then we must have $y=0$ and $z=0$. In this case we observe that $\widetilde{\alpha}([0:1:0])=(0, 0, 0)$ and $[0:1:0]\in L_\infty$. On the other hand, if $x>0$, then setting $[u:v:w]=[x:y+iz:0]$ gives
\[
\widetilde{\alpha}([x:y+iz:0])=\left(\frac{x^2}{x^2+y^2+z^2},\frac{x(y+iz)}{x^2+y^2+z^2}\right)=\left(\frac{x^2}{x},\frac{x(y+iz)}{x}\right)=(x,y+iz),
\]
where we used the fact that $x$, $y$, and $z$ are real. In either case, we have that every boundary point of $B$ is in the image of $L_\infty$.
\end{proof}

\section{Geometry of the image}\label{section:geometry}

For the remainder of the paper, except where noted, we let $f(u,v)\in\mathbb{C}[u,v]$ denote an irreducible polynomial of degree $n$, and $X=\mathbf{V}(f)\subset \mathbb{C}^2$. We further assume that $X$ is not a line through the origin. In this section we study the geometry of $\alpha(X)$ by analyzing how $X$ meets complex lines through the origin.

\subsection{Root counting and separating line segments}

\begin{defn}
For $(x,y,z)\in B$ and $\lambda\in\mathbb{C}$, define
\[
\beta(x,y,z,\lambda)=\left(\frac{x\lambda}{\sqrt{x-x^2-y^2-z^2}},\frac{(y+iz)\lambda}{\sqrt{x-x^2-y^2-z^2}}\right).
\]
Additionally, for convenience we introduce the notation
\[
R(x,y,z)=\frac{x}{\sqrt{x-x^2-y^2-z^2}}.
\]
\end{defn}

For a fixed point $(x,y,z)\in B$, the map $\lambda\mapsto \beta(x,y,z,\lambda)$ parameterizes the complex line through the origin whose projection under $\alpha$ contains $(x,y,z)$. By Theorem \ref{thm:uv_formula}, $R(x,y,z)$ is just the modulus of any $u$-coordinate in the preimage of $(x,y,z)$ under $\alpha$.

\begin{rem}\label{rem:comp_ident}
If $\lambda\in  S^1$, then
\[
\alpha\left(\beta(x,y,z,\lambda)\right)=(x,y,z).
\]
Indeed, by Corollary \ref{cor:circle_action}, $\alpha(\beta(x, y, z, \lambda))$ is independent of the choice of $\lambda\in S^1$. Letting $\lambda=1$ and then performing the same calculation as in the proof of Proposition \ref{prop:image_ball} yields the above equality.
\end{rem}

\begin{defn}
For $\mathbf{p}\in B$, define $N_f(\mathbf{p})$ to be the number of zeros, counted with multiplicity, of the polynomial
\[
\lambda\mapsto f\left(\beta(\mathbf{p},\lambda)\right)
\]
in the open unit disk $|\lambda|<1$.
\end{defn}

Because $f$ has degree $n$, the polynomial $f(\beta(\mathbf{p},\lambda))$ has degree at most $n$ in $\lambda$. It is not identically zero, because if it were, the entire complex line parameterized by $\beta(\mathbf{p},\lambda)$ would lie in $X$, and irreducibility would force $X$ itself to be a line through the origin, contradicting our assumption on $X$. Thus $N_f(\mathbf{p})\in\{0,1,\dots,n\}$.

\begin{prop}\label{prop:alternate_count}
For $\mathbf{p}=(x,y,z)\in B$, $N_f(\mathbf{p})$ is also equal to the number of roots, counted with multiplicity, of the polynomial
\[
u\mapsto f\left(u,\frac{y+iz}{x}u\right)
\]
in the disk $|u|<R(\mathbf{p})$.
\end{prop}

\begin{proof}
By definition,
\[
\beta(\mathbf{p},\lambda)=\left(R(\mathbf{p})\lambda,\frac{y+iz}{x}R(\mathbf{p})\lambda\right).
\]
Setting $u=R(\mathbf{p})\lambda$ converts the equation
\[
f\left(\beta(\mathbf{p},\lambda)\right)=0
\]
into
\[
f\left(u,\frac{y+iz}{x}u\right)=0,
\]
and the condition $|\lambda|<1$ is equivalent to $|u|<R(\mathbf{p})$.
\end{proof}

\begin{lemma}\label{lem:root_cont}
The function
\[
N_f: B\setminus \alpha(X)\rightarrow \{0,1,\dots,n\}
\]
is locally constant and therefore continuous.
\end{lemma}

\begin{proof}
For $\mathbf{p}\in B$ and $\lambda \in \mathbb{C}$, we define
\[
g_{\mathbf{p}}(\lambda)=f\left(\beta(\mathbf{p},\lambda)\right)=\sum_{j=0}^n a_j(\mathbf{p})\lambda^j.
\]
Now, fix some $\mathbf{p}\in B\setminus \alpha(X)$. We begin by showing that $g_{\mathbf{p}}(\lambda)$ is nonzero on $S^1$. If $g_{\mathbf{p}}(\lambda_0)=0$ for some $|\lambda_0|=1$, then $\beta(\mathbf{p},\lambda_0)\in X$ and, as noted in Remark \ref{rem:comp_ident},
\[
\alpha\left(\beta(\mathbf{p},\lambda_0)\right)=\mathbf{p},
\]
which would imply $\mathbf{p}\in\alpha(X)$, a contradiction. Hence, $g_{\mathbf{p}}(\lambda)$ has no zeros on $S^1$. Since $|g_{\mathbf{p}}(\lambda)|$ is continuous and nonzero on the unit circle, by compactness it attains a positive minimum value,
\[
\min_{|\lambda|=1}|g_{\mathbf{p}}(\lambda)|=m>0.
\]
By the continuity of the $a_j$, there exists some $\delta>0$ such that whenever $\|\mathbf{q}-\mathbf{p}\|<\delta$,
\[
\max_{0\le j\le n}|a_j(\mathbf{q})-a_j(\mathbf{p})|<\frac{m}{n+1}.
\]
When $|\lambda|=1$, this gives
\begin{align*}
|g_{\mathbf{q}}(\lambda)-g_{\mathbf{p}}(\lambda)|
&= \left|\sum_{j=0}^n \left(a_j(\mathbf{q})\lambda^j-a_j(\mathbf{p})\lambda^j\right)\right|\\
&\leq \sum_{j=0}^n \left|a_j(\mathbf{q})-a_j(\mathbf{p})\right||\lambda^j|\\
&< \sum_{j=0}^n \frac{m}{n+1}\\
&\leq |g_{\mathbf{p}}(\lambda)|.
\end{align*}
Rouch\'e's theorem then implies that $g_{\mathbf{q}}$ and $g_{\mathbf{p}}$ have the same number of zeros in $|\lambda|<1$. In addition, $g_{\mathbf{q}}$ cannot vanish on $S^1$. If $g_{\mathbf{q}}(\lambda_0)=0$ for some $\lambda_0\in S^1$, then the above inequality would give
\[
|g_{\mathbf{p}}(\lambda_0)|=|g_{\mathbf{p}}(\lambda_0)-g_{\mathbf{q}}(\lambda_0)|<|g_{\mathbf{p}}(\lambda_0)|,
\]
which is a contradiction. Thus $\mathbf{q}\notin\alpha(X)$ as well. Therefore, $N_f$ is constant on a neighborhood of $\mathbf{p}$ inside $B\setminus\alpha(X)$.
\end{proof}

The following theorem will be a primary tool for intersecting $\alpha(X)$ with rays in Section \ref{section:visualization}.

\begin{thm}\label{thm:line_segment}
Let $\mathbf{p},\mathbf{q}\in B$. If $N_f(\mathbf{p})\neq N_f(\mathbf{q})$, then the line segment joining $\mathbf{p}$ to $\mathbf{q}$ intersects $\alpha(X)$.
\end{thm}

\begin{proof}
Suppose the segment $\overline{\mathbf{p}\mathbf{q}}$ does not intersect $\alpha(X)$. Since $B$ is convex, the entire segment lies in $B\setminus \alpha(X)$. By Lemma \ref{lem:root_cont}, the function $N_f$ is locally constant on this segment. Because a locally constant function on a connected set is constant, $N_f$ must take the same value at $\mathbf{p}$ and $\mathbf{q}$, contradicting our hypothesis. Therefore, $\overline{\mathbf{p}\mathbf{q}}\cap \alpha(X)\neq \emptyset$.
\end{proof}

\subsection{Star-shaped domains}

\begin{defn}
A set $S\subset\mathbb{R}^m$ is \emph{star-shaped} if there exists a point $\mathbf{p}\in S$ such that, for every $\mathbf{q}\in S$, the line segment joining $\mathbf{p}$ to $\mathbf{q}$ is contained in $S$. Such a point $\mathbf{p}$ is called a \emph{base point} of $S$.
\end{defn}

Although $\alpha(X)$ is, in general, not star-shaped, the root-counting function produces a family of nested sets that become star-shaped after adjoining $(0,0,0)$. Under modest restrictions, the union of these boundaries is precisely the image of the closure of $X$ in $\cp$ under $\widetilde{\alpha}$ (Theorem \ref{thm:projective_image}).

\begin{defn}
For each $k\in\{1,\dots,n\}$, define
\[
S_f(k)=\{\mathbf{p}\in B:N_f(\mathbf{p})<k\}.
\]
\end{defn}

\begin{prop}\label{prop:monotonicity}
Let $\mathbf{p}\in B$ and $t\in(0,1)$. Then $N_f(t\mathbf{p})\le N_f(\mathbf{p})$. If, in addition, $t\mathbf{p}\in\alpha(X)$, then the inequality is strict.
\end{prop}

\begin{proof}
Write $\mathbf{p}=(x,y,z)$. By Proposition \ref{prop:alternate_count}, both $N_f(\mathbf{p})$ and $N_f(t\mathbf{p})$ count roots of the same polynomial
\[
u\mapsto f\left(u,\frac{y+iz}{x}u\right),
\]
because scaling $\mathbf{p}$ by $t$ does not change the ratio $(y+iz)/x$. The only difference is the radius of the disk. Namely,
\[
R(t\mathbf{p})^2=\frac{t^2x^2}{tx-t^2(x^2+y^2+z^2)}=\frac{tx^2}{x-t(x^2+y^2+z^2)}.
\]
Since $\mathbf{p}\in B$, we have
\[
0< x-(x^2+y^2+z^2) < x-t(x^2+y^2+z^2),
\]
and therefore
\[
R(t\mathbf{p})^2=\frac{tx^2}{x-t(x^2+y^2+z^2)}<\frac{x^2}{x-(x^2+y^2+z^2)}=R(\mathbf{p})^2.
\]
Thus every root counted by $N_f(t\mathbf{p})$ is also counted by $N_f(\mathbf{p})$, with multiplicity. If $t\mathbf{p}\in\alpha(X)$, then there is a root on the boundary of the disk $|u|<R(t\mathbf{p})$ which would not be counted by $N_f(t\mathbf{p})$, but would be counted by $N_f(\mathbf{p})$. Therefore, in this case, the inequality is strict.
\end{proof}

\begin{cor}
For each $k\in\{1,\dots,n\}$, the set $S_f(k)\cup\{(0,0,0)\}$ is star-shaped with base point $(0,0,0)$.
\end{cor}

\begin{proof}
If $S_f(k)$ is empty, then the statement is trivially true. Suppose not and fix some $\mathbf{p}\in S_f(k)$. By Proposition \ref{prop:monotonicity},
\[
N_f(t\mathbf{p})\le N_f(\mathbf{p})<k,
\]
for any $t\in(0,1)$. Hence the entire line segment from $(0,0,0)$ to $\mathbf{p}$ lies in $S_f(k)\cup\{(0,0,0)\}$. Since $\mathbf{p}\in S_f(k)$ was arbitrary, $S_f(k)\cup\{(0,0,0)\}$ is star-shaped with base point $(0,0,0)$.
\end{proof}

\begin{lemma}\label{lem:boundary_points}
The points of $\bigcup_{k=1}^n \partial S_f(k)$ that lie in $B$ are exactly $\alpha(X\cap\A)$.
\end{lemma}

\begin{proof}
Let $\mathbf{p}\in B\cap \partial S_f(k)$ and suppose $\mathbf{p}\notin \alpha(X)$. Then $\mathbf{p}\in B\setminus \alpha(X)$, so Lemma \ref{lem:root_cont} gives a neighborhood $U\subset B\setminus\alpha(X)$ on which $N_f$ is constant, say identically equal to $j$.

If $j<k$, then $U\subset S_f(k)$, so $\mathbf{p}$ is an interior point of $S_f(k)$, contradicting our assumption. If $j\ge k$, then $U\cap S_f(k)=\emptyset$, so $U$ is an open neighborhood of $\mathbf{p}$ disjoint from $S_f(k)$. Hence $\mathbf{p}\notin \partial S_f(k)$, again a contradiction. Therefore, $\mathbf{p}\in\alpha(X)$.

On the other hand, let $\mathbf{p}\in \alpha(X\cap\A)$ and $U$ be any open set containing $\mathbf{p}$. Since $X$ is not a line through the origin, the ray from $(0,0,0)\in\mathbb{R}^3$ through $\mathbf{p}$ intersects $\alpha(X)$ in finitely many points. Therefore, we can choose some $t\in (0,1)$ so that $t^{-1}\mathbf{p}\in U\setminus \alpha(X)$. By Proposition \ref{prop:monotonicity},
\[
N_f(\mathbf{p})=N_f(t t^{-1}\mathbf{p}) < N_f(t^{-1}\mathbf{p}).
\]
Note that the above inequality implies $N_f(\mathbf{p})+1 \in\{1,\ldots,n\}$. Since $N_f$ is locally constant on $B\setminus \alpha(X)$, there is some neighborhood $V$ of $t^{-1}\mathbf{p}$ where $N_f$ is constant and strictly greater than $N_f(\mathbf{p})$. It follows that
\[
V\subset S_f(N_f(\mathbf{p})+1)^c.
\]
In particular, $t^{-1}\mathbf{p} \notin S_f(N_f(\mathbf{p})+1)$. By definition, $\mathbf{p}\in S_f(N_f(\mathbf{p})+1)$, so $U$ contains points both in $S_f(N_f(\mathbf{p})+1)$ and $S_f(N_f(\mathbf{p})+1)^c$. Because $U$ was arbitrary, $\mathbf{p}\in\partial S_f(N_f(\mathbf{p})+1)$.
\end{proof}

\begin{thm}\label{thm:projective_image}
Let $Z\subset \cp$ be an irreducible projective curve that is not a line through $[0:0:1]$ and is not the line at infinity. If $f(u,v)=0$ is a defining equation for the affine curve $X=Z\setminus \mathbf{V}(w)$, then
\[
\widetilde{\alpha}(Z) = \bigcup_{k=1}^n \partial S_f(k).
\]
\end{thm}

\begin{proof}
We begin by noting that $(0,0,0)$  is trivially in the union of boundaries as well as the image of $Z$. Therefore, we omit this case from the remainder of the proof. It follows from our hypothesis that $X\cap\A$ is dense in $Z$, so continuity of $\widetilde{\alpha}$ implies
\[
\overline{\widetilde{\alpha}(X\cap\A)} = \widetilde{\alpha}(Z).
\]
However, by Lemma \ref{lem:boundary_points}, the union of the boundaries contains $\widetilde{\alpha}(X\cap\A)$. Since the boundaries are closed,
\[
\widetilde{\alpha}(Z) \subset \bigcup_{k=1}^n \partial S_f(k).
\]

Suppose now that $\mathbf{p}\in \partial S_f(k)\cap \partial B$ for some $k$ and that $\mathbf{p}\notin\widetilde{\alpha}(Z)$. Since $\widetilde{\alpha}(Z)$ is closed, there exists some ball $B_\varepsilon(\mathbf{p})$ such that $B_\varepsilon(\mathbf{p})\cap\widetilde{\alpha}(Z)=\emptyset$. Restricting to $B$ and applying Lemma \ref{lem:boundary_points}, we have
\begin{align*}
\left(B_\varepsilon(\mathbf{p})\cap B\right) \cap \left(\partial S_f(k) \cap B\right) &\subset \left(B_\varepsilon(\mathbf{p})\cap B\right)\cap \widetilde{\alpha}(X\cap\A)\\
&= \left(B_\varepsilon(\mathbf{p})\cap B\right) \cap \left(\widetilde{\alpha}(Z)\cap B\right)\\
&=\emptyset.
\end{align*}
However,
\[
B\setminus \partial S_f(k) = \left(\interior(S_f(k))\cap B\right)\sqcup\left(\interior(S_f(k)^c)\cap B\right).
\]
Since $B_\varepsilon(\mathbf{p})\cap B$ is connected, either
\[
B_\varepsilon(\mathbf{p})\cap B \subset \interior(S_f(k))\cap B
\]
or
\[
B_\varepsilon(\mathbf{p})\cap B \subset \interior(S_f(k)^c)\cap B.
\]
Suppose $B_\varepsilon(\mathbf{p})\cap B \subset \interior(S_f(k))\cap B$. The open line segment parameterized by $(1-\varepsilon t)\mathbf{p}$ for $t\in(0,1)$ is contained within $B_\varepsilon(\mathbf{p})\cap B$ and hence within $S_f(k)$. Since $S_f(k)\cup \{(0,0,0)\}$ is star-shaped, the entire open line segment from $(0,0,0)$ to $\mathbf{p}$ is also a subset of $S_f(k)$. Fix any $(x,y,z)$ on this line segment. By Proposition \ref{prop:alternate_count}, the polynomial
\[
u\mapsto f\left(u,\frac{y+iz}{x}u\right)
\]
has fewer than $k\leq n$ zeros on the open disk $|u|<R(x,y,z)$. However, the polynomial is unchanged by scaling $(x,y,z)$, while $R(x,y,z)$ tends toward infinity as $(x,y,z)$ approaches $\mathbf{p}$ along this line segment. In particular, the line parameterized by $u$ and defined by
\[
u\mapsto \left(u,\frac{y+iz}{x}u\right)
\]
has fewer than $n$ intersections with $X$, counted with multiplicity. By B\'{e}zout's theorem, the projective completion of this line must intersect $Z$ on $L_\infty$. However, the image of this projective line under $\widetilde{\alpha}$ is exactly the line segment connecting $(0,0,0)$ to $\mathbf{p}$. The point at infinity of this line maps to $\mathbf{p}$, so $\mathbf{p}\in\widetilde{\alpha}(Z)$, a contradiction.

Suppose instead that $B_\varepsilon(\mathbf{p})\cap B \subset \interior(S_f(k)^c)\cap B$. Then
\[
B_\varepsilon(\mathbf{p})\cap B \subset \interior(S_f(k)^c) \subset S_f(k)^c.
\]
Since $S_f(k)\subset B$, $B_\varepsilon(\mathbf{p})\subset S_f(k)^c$. In particular, $\mathbf{p}$ is not a boundary point of $S_f(k)$, again a contradiction.

In either case, every point of $\partial S_f(k)\cap \partial B$ lies in $\widetilde{\alpha}(Z)$. Combined with Lemma \ref{lem:boundary_points}, this proves the reverse inclusion and hence the theorem.
\end{proof}

\begin{figure}[!htb]
    \centering
    \includegraphics[height=1.5in]{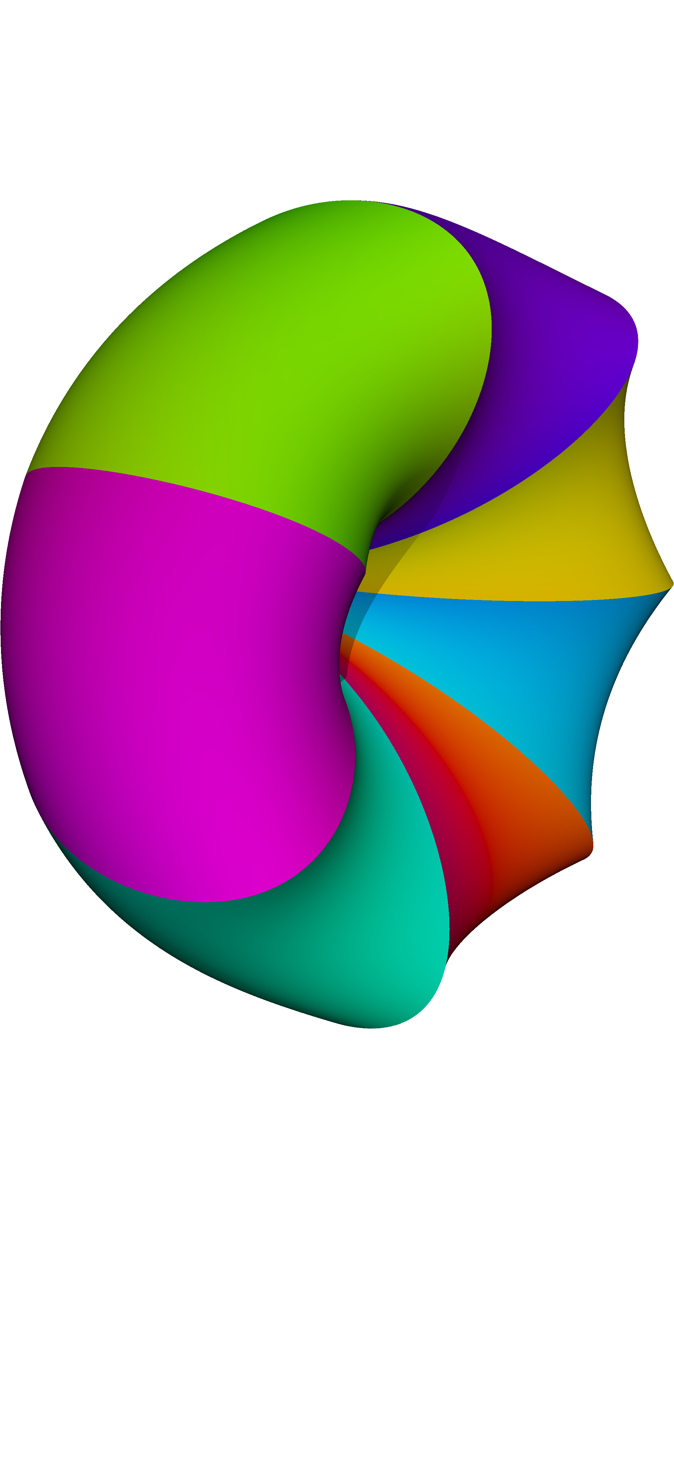}\hfill
    \includegraphics[height=1.5in]{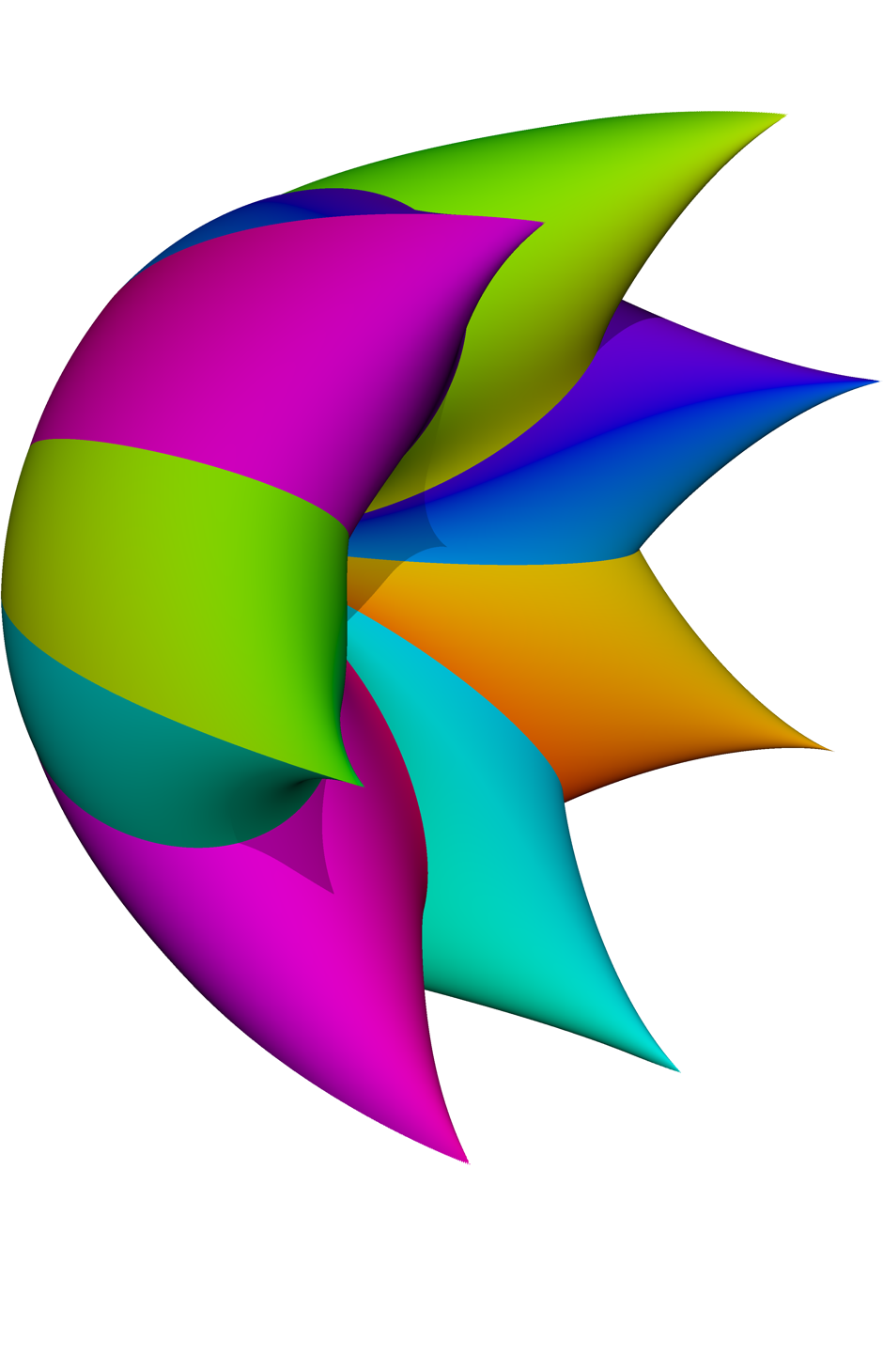}\hfill
    \includegraphics[height=1.5in]{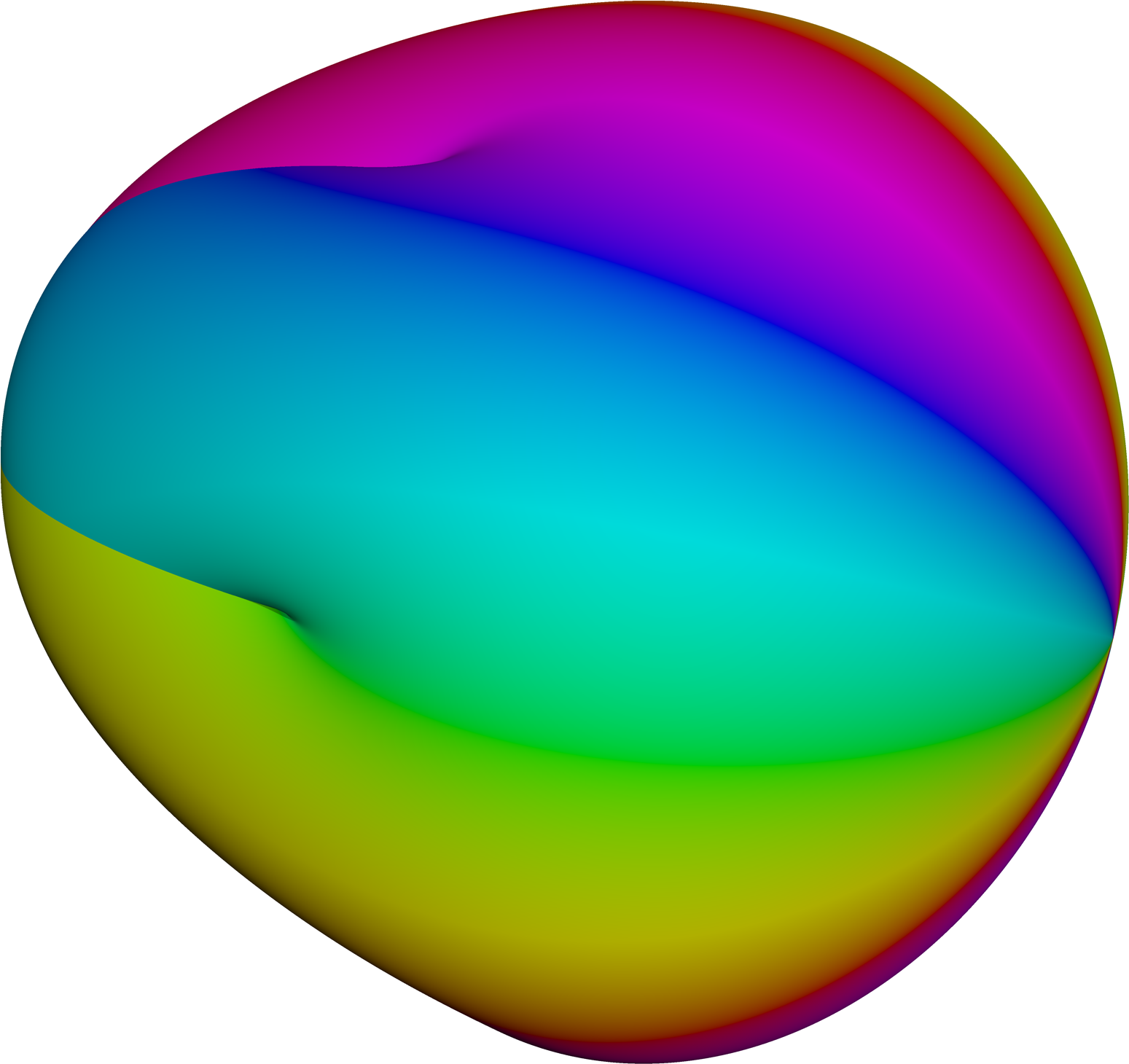}\hfill
    \includegraphics[height=1.5in]{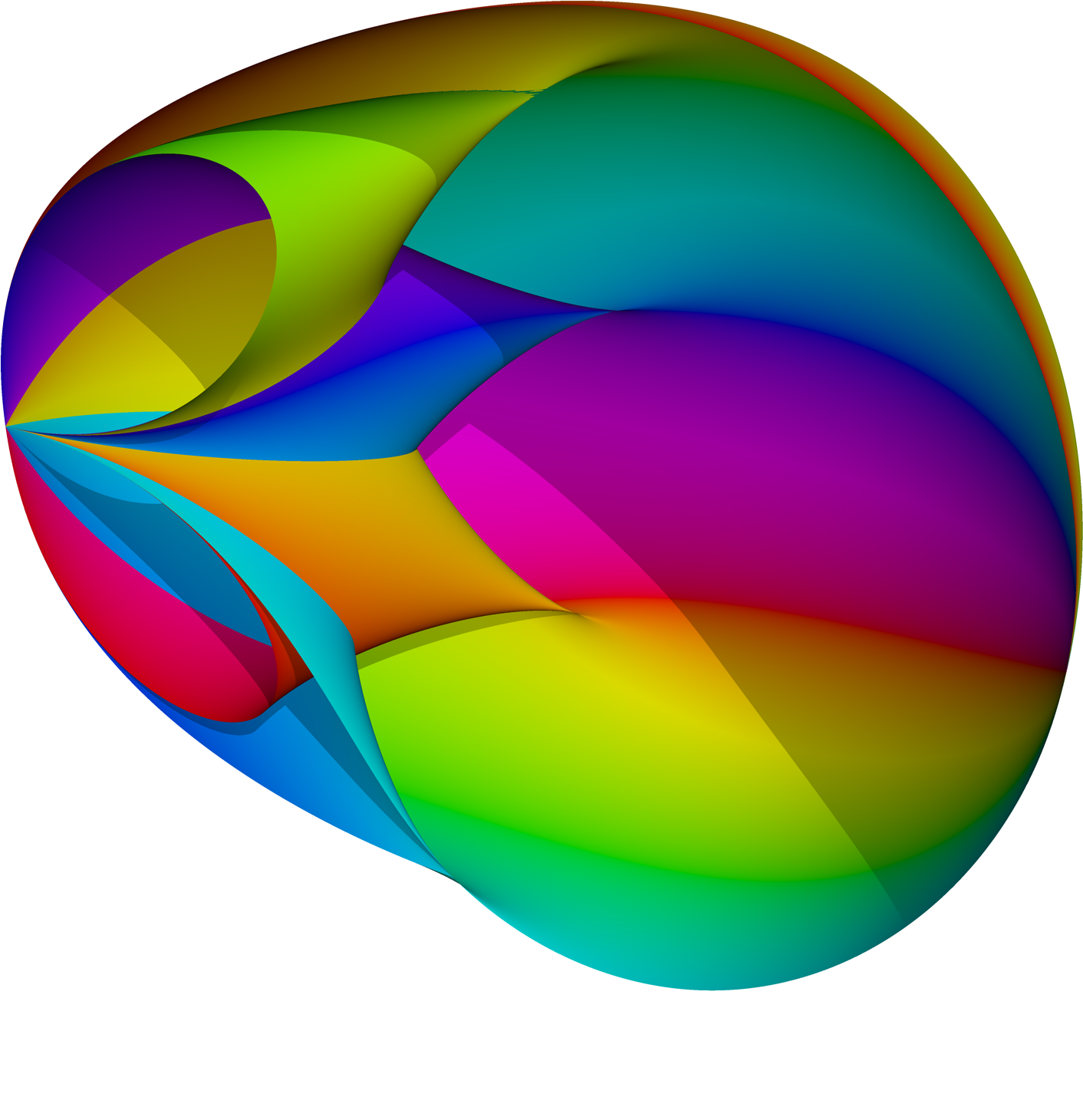}
    \caption{Ray-traced views of $\partial S_f(2)$, $\partial S_f(3)$, $\partial S_f(4)$, and a cutaway view of their union, for a transformed Klein quartic with $f(u,v)=uv^3+u^3+v$.}
\end{figure}

\subsection{Automorphisms and symmetry}

There exists a subgroup of $\operatorname{PGL}(3,\mathbb{C})$ whose elements descend to rotations of $\overline{B}$ under $\widetilde{\alpha}$ (Theorem \ref{thm:diagonal_automorphism}). Every finite-order element of $\operatorname{PGL}(3,\mathbb{C})$ is conjugate to an element of this subgroup (Proposition \ref{prop:diagonalize}). Consequently, if $Z\subset \cp$ is a projective curve with a finite-order automorphism $\sigma: Z\rightarrow Z$ that extends to $\cp$, then, after a change of coordinates, $\sigma$ descends to a rotation of the image of $Z$ in $\mathbb{R}^3$.

Automorphisms of a projective curve do not, in general, extend to $\cp$. However, if $Z$ is smooth and has degree at least four, then all automorphisms do extend to $\cp$ (Theorem \ref{thm:degree_four}). Additionally, for a smooth projective curve $Z$ and any finite-order automorphism $\sigma: Z\rightarrow Z$, there exists a map $Z\rightarrow \mathbb{R}^3$ such that $\sigma$ descends to a rotation of the image of $Z$. Therefore, finite-order automorphisms of projective curves can always be visualized as rotational symmetries of their images in $\mathbb{R}^3$ (Theorem \ref{thm:Kummer_morphism}).

\begin{thm}\label{thm:diagonal_automorphism}
Let $\sigma: \cp\rightarrow \cp$ be given by $[u:v:w]\mapsto [e^{i\theta_u}u:e^{i\theta_v}v:w]$, where $\theta_u$, $\theta_v\in\mathbb{R}$. If $\widetilde{\alpha}([u:v:w])=(x,y,z)$, then
\[
\widetilde{\alpha}\left(\sigma([u:v:w])\right)=
\begin{pmatrix}
1 & 0 & 0\\
0 & \cos(\theta_v-\theta_u) & -\sin(\theta_v-\theta_u)\\
0 & \sin(\theta_v-\theta_u) & \cos(\theta_v-\theta_u)
\end{pmatrix}
\begin{pmatrix}
x\\
y\\
z
\end{pmatrix}.
\]
\end{thm}

\begin{proof}
Identifying $\mathbb{R}^3$ with $\mathbb{R}\times\mathbb{C}$, a direct computation gives
\begin{align*}
\widetilde{\alpha}\left(\sigma([u:v:w])\right)
&=\left(\frac{|ue^{i\theta_u}|^2}{|w|^2+|ue^{i\theta_u}|^2+|ve^{i\theta_v}|^2},
\frac{\overline{ue^{i\theta_u}}\,ve^{i\theta_v}}{|w|^2+|ue^{i\theta_u}|^2+|ve^{i\theta_v}|^2}\right)\\
&=\left(\frac{|u|^2}{|w|^2+|u|^2+|v|^2},\frac{\bar{u}v\,e^{i(\theta_v-\theta_u)}}{|w|^2+|u|^2+|v|^2}\right)\\
&=\left(x,(y+iz)e^{i(\theta_v-\theta_u)}\right).
\end{align*}
The second coordinate is a rotation by the angle $\theta_v-\theta_u$ of $y+iz$ in the $yz$-plane, which is exactly the rotation about the $x$-axis as described in the statement of the theorem.
\end{proof}

\begin{exmp}
Consider the Klein quartic $u^3v+v^3w+uw^3=0$, which admits the order-seven automorphism $[u:v:w]\mapsto [e^{6\pi i/7}u:e^{2\pi i/7}v:w]$. By Theorem \ref{thm:diagonal_automorphism}, this automorphism becomes a rotation of the image by $-4\pi/7$ about the $x$-axis. To better visualize the automorphism, we use the argument of $u$ to assign a hue at each point of the image.

\begin{figure}[!htb]
    \centering
    \begin{subfigure}{0.3\textwidth}
        \centering
        \includegraphics[height=1.5in]{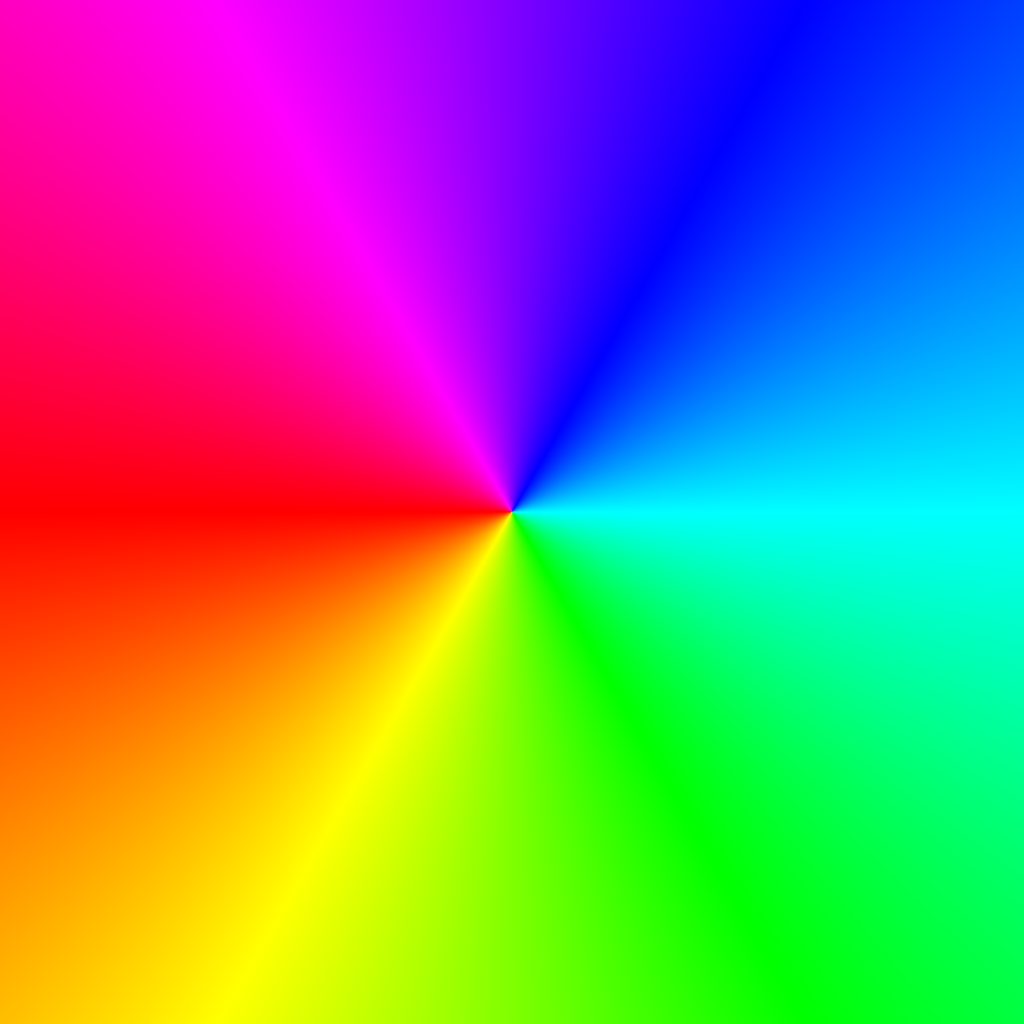}
        \caption{Hue mapping of the $u$-plane}
    \end{subfigure}
    \hfill
    \begin{subfigure}{0.3\textwidth}
        \centering
        \includegraphics[height=1.5in]{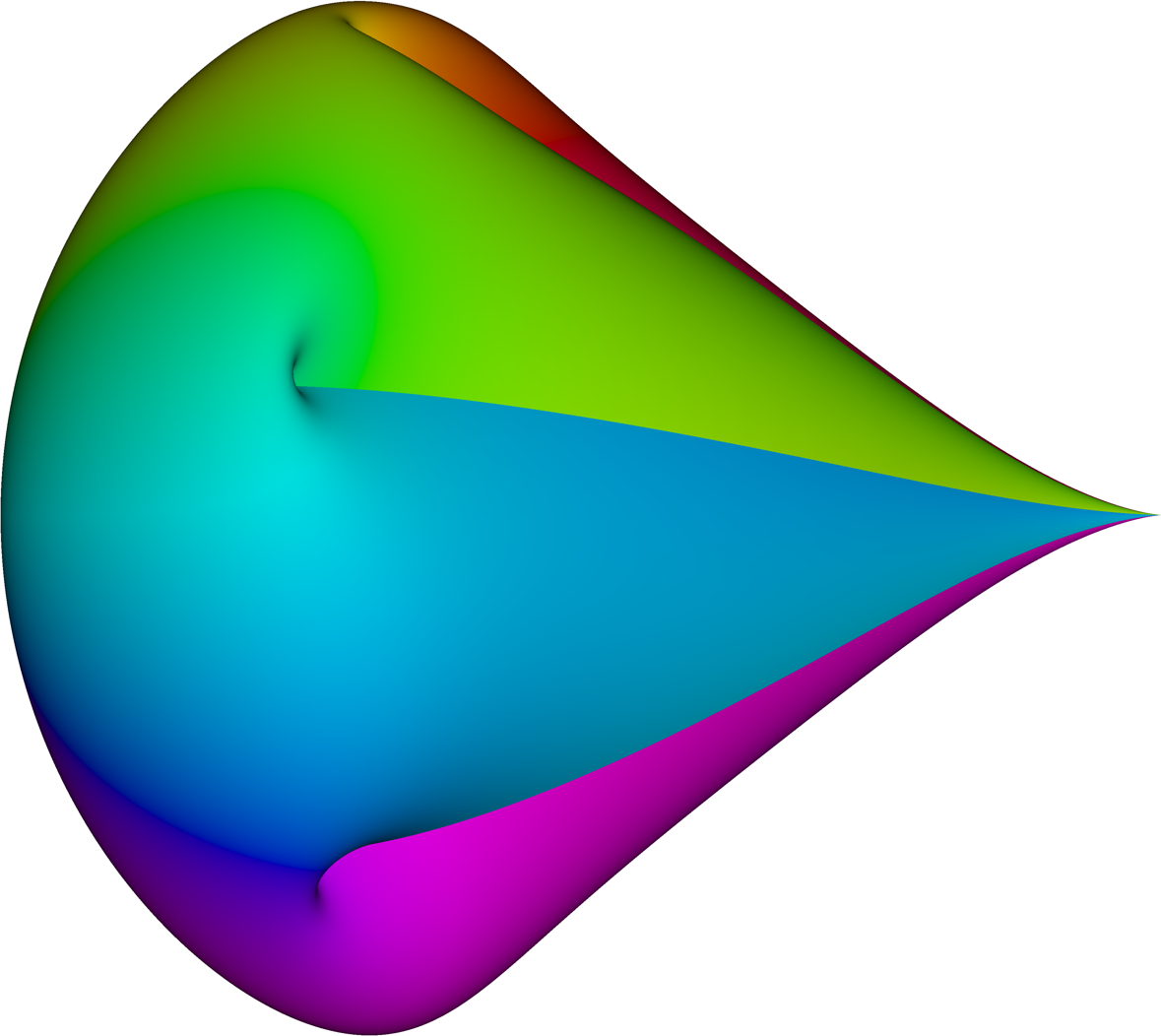}
        \caption{Side view}
    \end{subfigure}
    \hfill
    \begin{subfigure}{0.3\textwidth}
        \centering
        \includegraphics[height=1.5in]{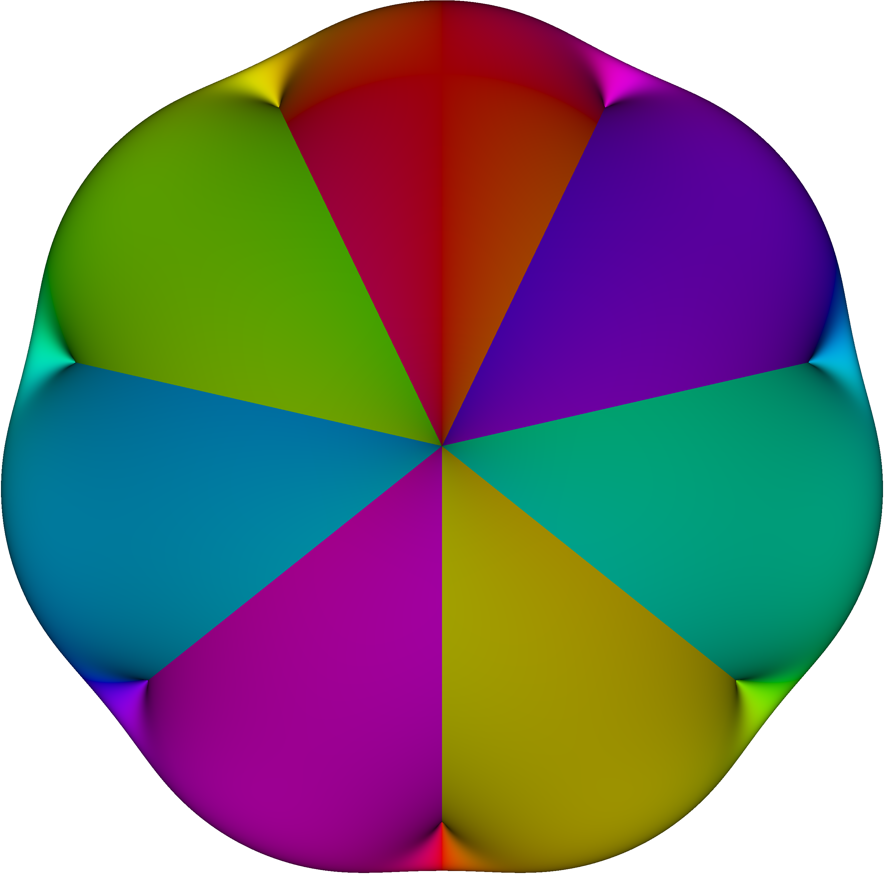}
        \caption{Front view}
    \end{subfigure}
    \caption{Domain-colored views of the Klein quartic. Rotation by $-4\pi/7$ results in a hue shift of $6\pi/7$, which is precisely the value of $\theta_u$ in the automorphism.}

\end{figure}
\end{exmp}

When $\theta_u=0$, the rotation does not shift the hue.

\begin{exmp}
Consider the transformed Fermat curve $v^7+(u-w)^7+w^7=0$ along with the automorphism $[u:v:w]\mapsto[u:e^{2\pi i/7}v:w]$. Note that this curve is projectively equivalent to the curve in Figure \ref{fig:intro}.

\begin{figure}[!htb]
    \centering
    \begin{subfigure}{0.3\textwidth}
        \centering
        \includegraphics[height=1.5in]{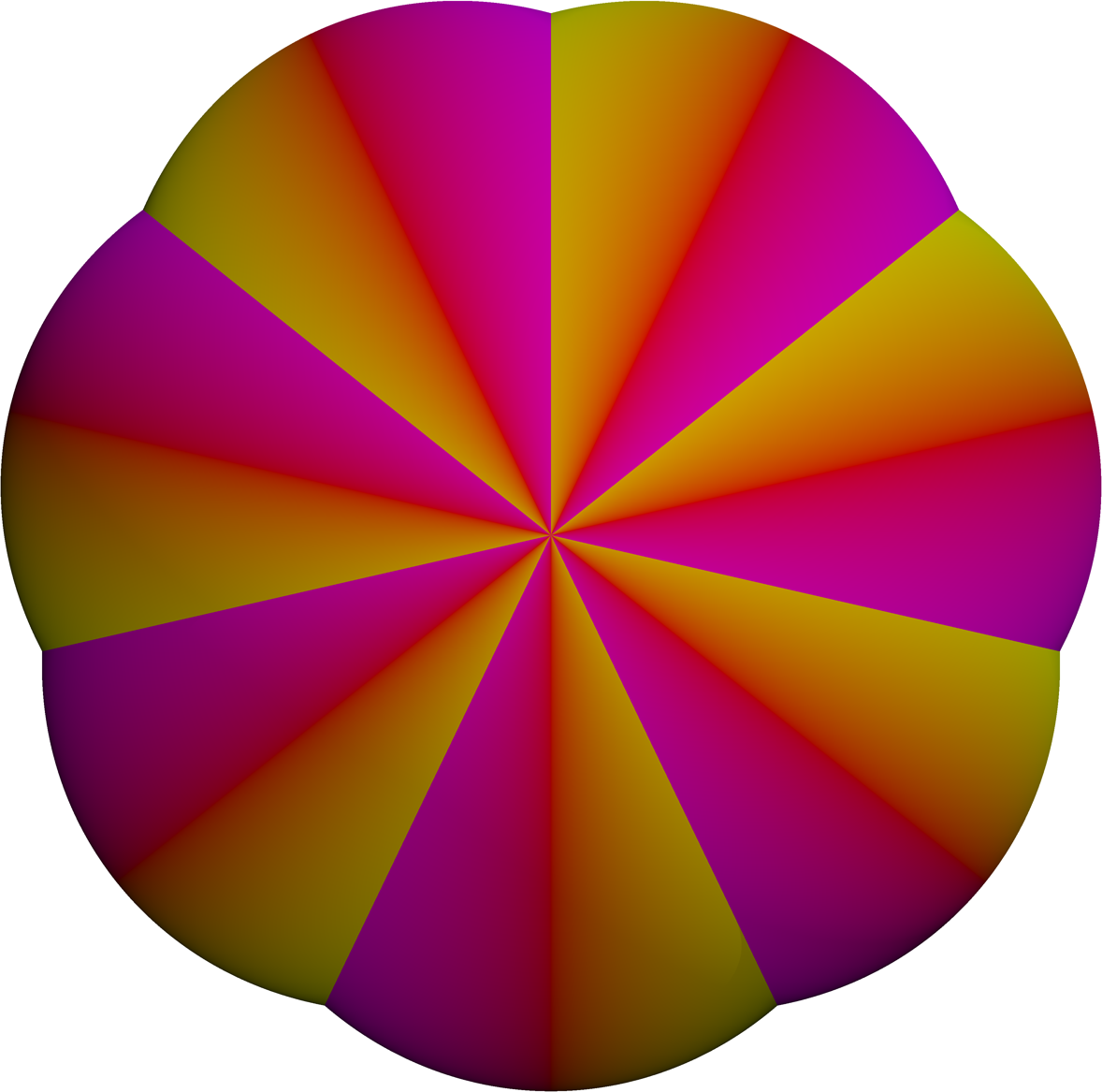}
        \caption{Rear view}
    \end{subfigure}
    \hfill
    \begin{subfigure}{0.3\textwidth}
        \centering
        \includegraphics[height=1.5in]{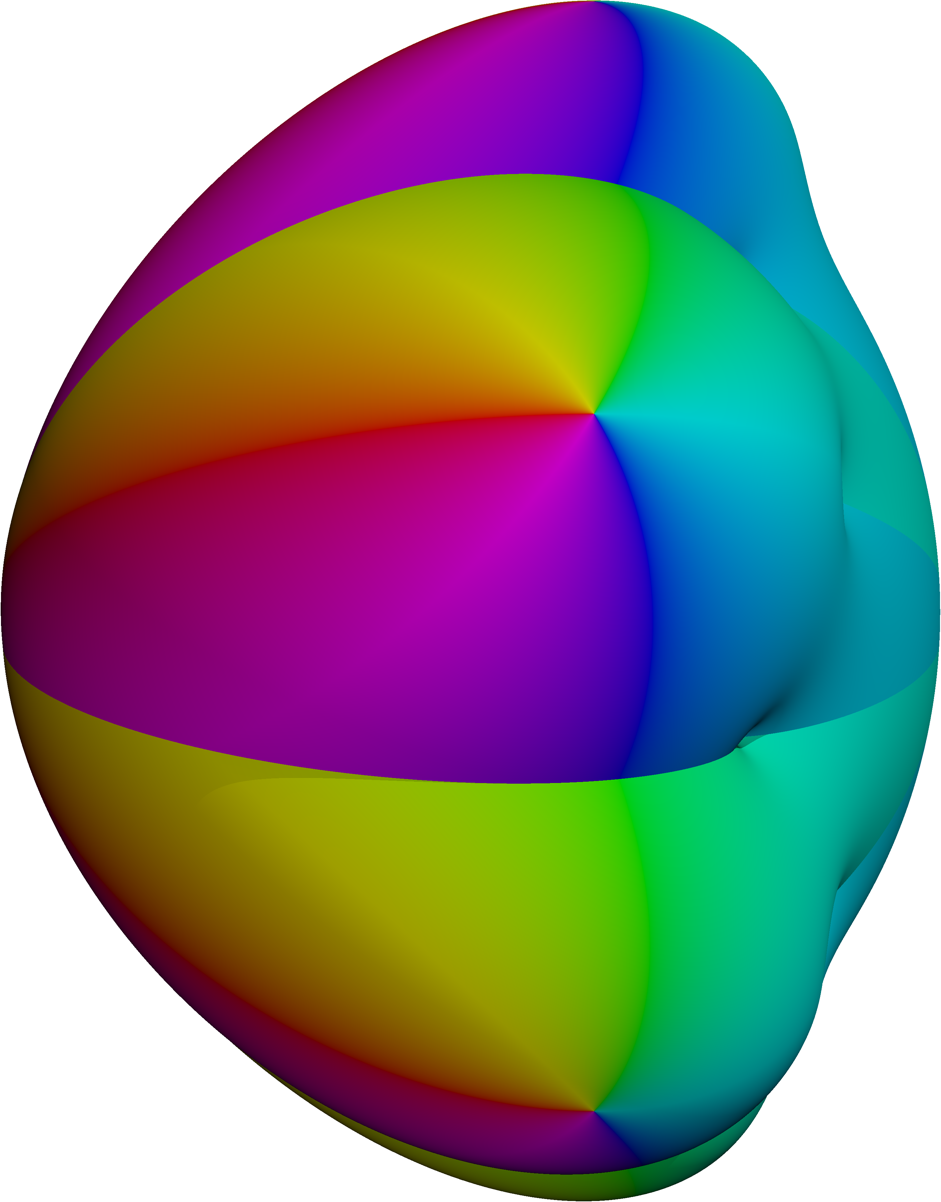}
        \caption{Side view}
    \end{subfigure}
    \hfill
    \begin{subfigure}{0.3\textwidth}
        \centering
        \includegraphics[height=1.5in]{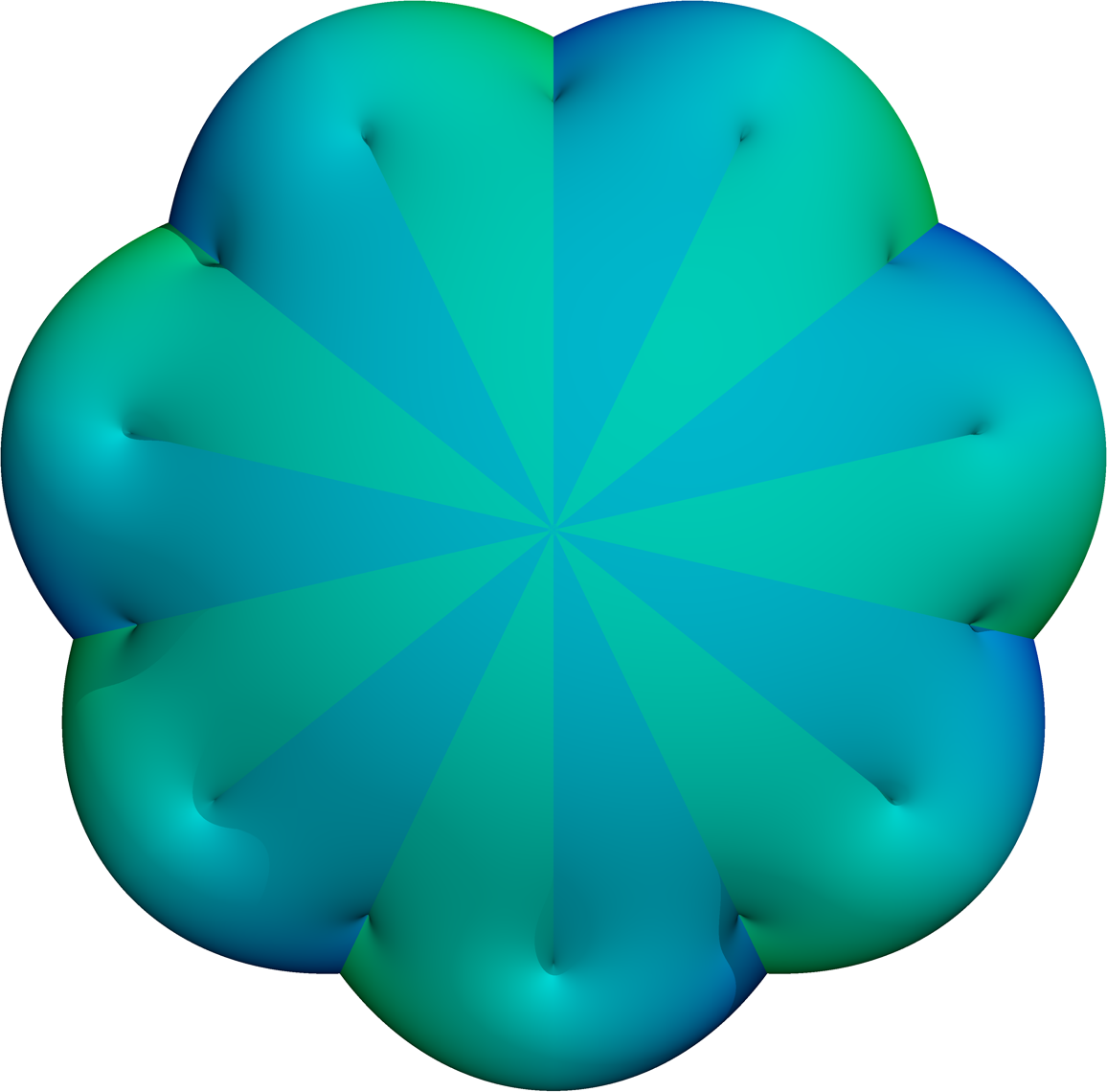}
        \caption{Front view}
    \end{subfigure}
    \caption{Domain-colored views of the curve $v^7+(u-w)^7+w^7=0$. Since $\theta_u=0$, there is no hue shift caused by rotation.}
\end{figure}
\FloatBarrier
\end{exmp}

If $\theta_u=\theta_v$, the automorphism descends to the identity map. In this case, we get a many-to-one map.

\begin{exmp}
Consider $u^6+u^3v^3+v^6+w^6=0$ with the order-$6$ automorphism $\sigma:[u:v:w]\rightarrow[e^{2\pi i/6}u:e^{2\pi i/6}v:w]$. Since $\sigma$ descends to the identity map on $\overline{B}$, on the affine chart $w=1$, the restriction $\alpha|_X$ is six-to-one.

\begin{figure}[!htb]
    \centering
    \begin{subfigure}{0.3\textwidth}
        \centering
        \includegraphics[height=1.5in]{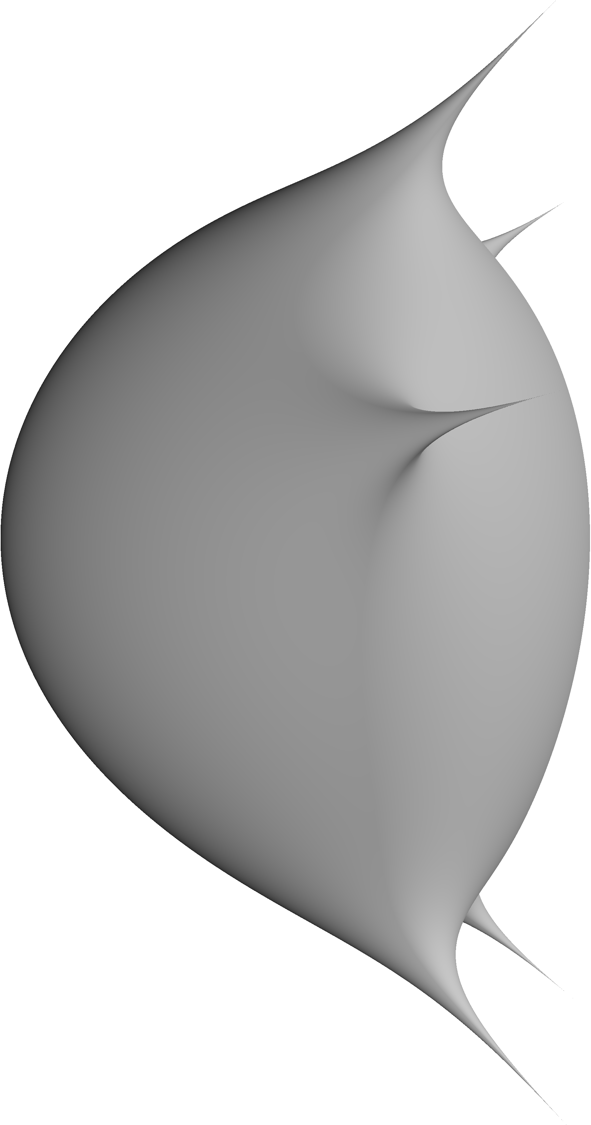}
        \caption{Side view}
    \end{subfigure}
    \hfill
    \begin{subfigure}{0.3\textwidth}
        \centering
        \includegraphics[height=1.5in]{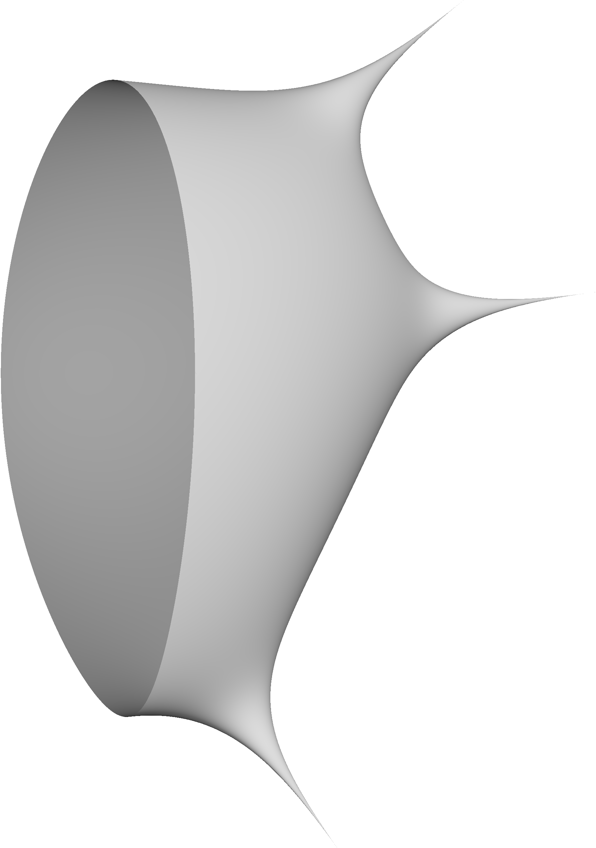}
        \caption{Cutaway view}
    \end{subfigure}
    \hfill
    \begin{subfigure}{0.3\textwidth}
        \centering
        \includegraphics[height=1.5in]{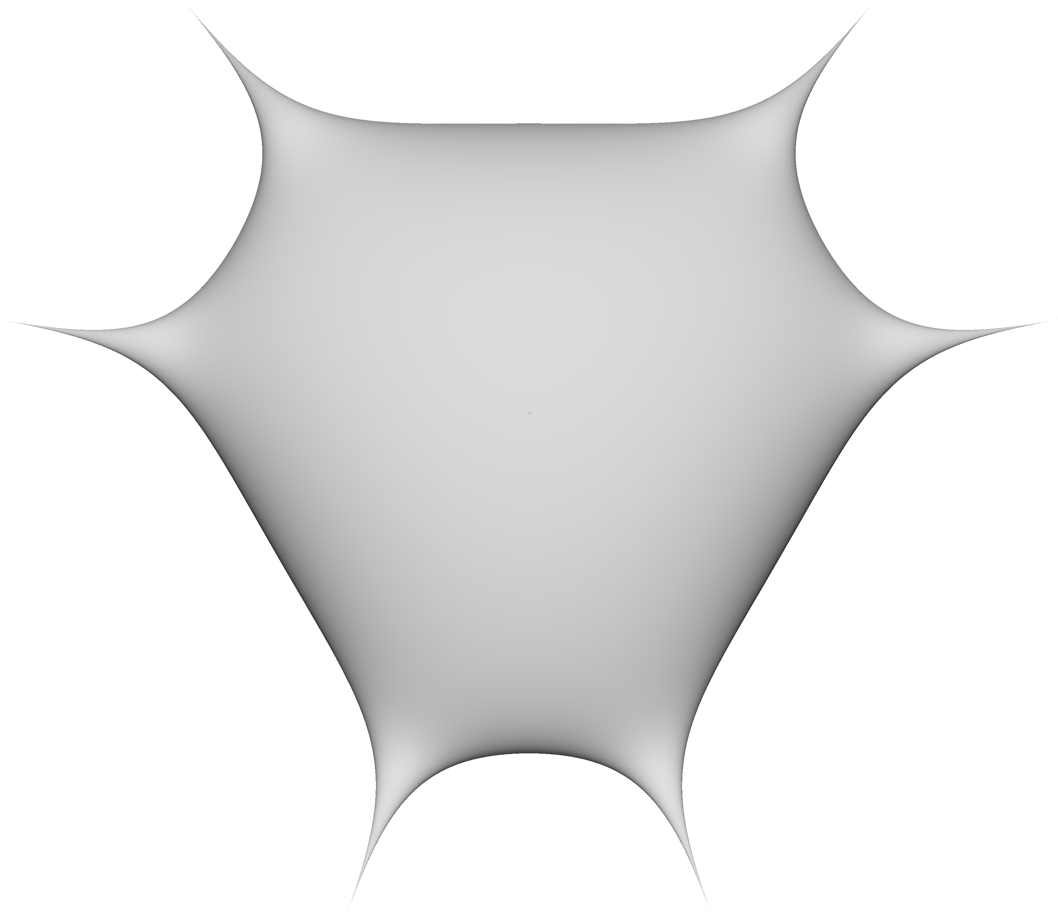}
        \caption{Front view}
    \end{subfigure}
    \caption{Views of the curve $u^6+u^3v^3+v^6+w^6=0$. Because every point of the affine image has six preimages, domain coloring is not being applied.}
    \label{fig:six_to_one}
\end{figure}

\end{exmp}

\begin{prop}\label{prop:diagonalize}
Let $\gamma:\cp\rightarrow\cp$ be a finite-order automorphism. Then there exists a projective change of coordinates $\psi:\cp\rightarrow\cp$ such that
\[
\gamma = \psi \circ \sigma\circ \psi^{-1},
\]
where $\sigma([u:v:w])=[e^{i\theta_u}u:e^{i\theta_v}v: w]$ for some $\theta_u$, $\theta_v\in\mathbb{R}$.
\end{prop}

\begin{proof}
Let $T$ be a representative of $\gamma$. Then there exists some positive integer $n$ such that $T^n = \lambda I$. Therefore, $T$ is diagonalizable and we can write $T=PDP^{-1}$ with
\[
D = \begin{pmatrix}
\lambda_1 & 0 & 0\\
0 & \lambda_2 & 0\\
0 & 0 & \lambda_3
\end{pmatrix}.
\]
Let $\sigma = [D]$ and $\psi = [P]$. Since $T^n=\lambda I$, $\lambda_i^n = \lambda$ and so $|\lambda_1|=|\lambda_2|=|\lambda_3|\neq 0$. In particular, $\lambda_1/\lambda_3$ and $\lambda_2/\lambda_3$ are on the unit circle and can be written as $e^{i\theta_u}$ and $e^{i\theta_v}$, respectively. Therefore, $\sigma([u:v:w]) = [e^{i\theta_u}u:e^{i\theta_v}v: w]$ as claimed.
\end{proof}

As a consequence of the previous proposition, if $\sigma:Z\rightarrow Z$ is a finite-order automorphism that extends to $\cp$, then after a projective change of coordinates we are precisely in the setting of Theorem \ref{thm:diagonal_automorphism}, in which case the automorphism can be visualized as a rotational symmetry. This raises the question of when automorphisms of curves extend to $\cp$. We do not handle the extension problem in full generality, but Theorem \ref{thm:degree_four} shows that every automorphism of a smooth plane curve of degree at least $4$ extends to a projective automorphism of $\cp$. This is a consequence of the following special case of Noether's theorem.

\begin{thm}[Noether]
Suppose $Z\subset\cp$ is a smooth plane curve of degree $n\geq 4$. If $\mathcal{L}$ is a line bundle on $Z$ satisfying $\deg\mathcal{L}=n$ and $h^0(Z,\mathcal{L})\geq 3$, then $\mathcal{L}\simeq\mathcal{O}_Z(1)$.
\end{thm}

For a modern proof and generalization of Noether's theorem, see Hartshorne \cite[Theorem 2.1]{hartshorne1986generalized}. The above statement is the special case when $r=1$ and $e=0$.

\begin{thm}\label{thm:degree_four}
Suppose $i:Z\hookrightarrow\cp$ embeds $Z$ as a smooth plane curve of degree $n\geq 4$, and $\gamma:Z\rightarrow Z$ is an automorphism. Then there exists a projective change of coordinates $\psi:\cp\rightarrow\cp$ such that
\[
\psi\circ i=i\circ\gamma.
\]
\end{thm}

\begin{proof}
Define $j=i\circ\gamma$, $\mathcal{L}=i^*\mathcal{O}_{\cp}(1)$, and $\mathcal{M}=j^*\mathcal{O}_{\cp}(1)$. Since $\mathcal{M}\simeq\gamma^*\mathcal{L}$, the line bundle $\mathcal{M}$ has degree $n$. The three coordinate sections defining $j$ are linearly independent because $j(Z)=i(Z)$ is not contained in a line. Therefore $h^0(Z,\mathcal{M})\geq 3$. By the above special case of Noether's theorem, $\mathcal{M}\simeq\mathcal{O}_Z(1)\simeq\mathcal{L}$.

Since $Z$ is a plane hypersurface, the restriction sequence gives $h^0(Z,\mathcal{O}_Z(1))=3$. Fix an isomorphism $\mathcal{M}\simeq\mathcal{L}$. Under this identification of global sections, the coordinate sections defining $i$ and $j$ form two bases of $H^0(Z,\mathcal{L})$. They differ only by an element of $\operatorname{GL}(3,\mathbb{C})$, which gives rise to a projective change of coordinates $\psi:\cp\rightarrow\cp$ satisfying $j=\psi\circ i$. Since $j=i\circ\gamma$, the theorem follows.
\end{proof}

The preceding theorem only shows that automorphisms extend to $\cp$, but in order to apply Proposition \ref{prop:diagonalize} we also need that the automorphism has finite order. However, any smooth plane curve of degree at least four has genus at least three. By Hurwitz's theorem, the automorphism group is finite and therefore all automorphisms have finite order. Consequently, we can always apply Proposition \ref{prop:diagonalize} and Theorem \ref{thm:degree_four} to any automorphism of a smooth plane curve of degree at least four.

Despite this, having a smooth embedding in $\cp$ is the exception rather than the rule for higher genus curves. Nonetheless, for a smooth projective curve with a finite-order automorphism $\sigma$, it is still possible to find a morphism to $\cp$ such that Theorem \ref{thm:diagonal_automorphism} applies and the automorphism can be visualized.

\begin{thm}\label{thm:Kummer_morphism}
Let $Z$ be a smooth projective curve and $\sigma:Z\rightarrow Z$ be an automorphism of finite order $n$. Then there exists a nonconstant morphism $\phi:Z\rightarrow\cp$ and a primitive $n$-th root of unity $\zeta\in\mathbb{C}$ such that $\phi\circ\sigma = \tau\circ\phi$, where $\tau([u:v:w]) = [\zeta u:v:w]$.
\end{thm}

\begin{proof}
Let $L$ be the function field of $Z$ and $K = L^{\sigma^*}$ be the subfield fixed by the automorphism $\sigma^*:L\rightarrow L$. Since $\sigma^*$ has order $n$, the extension $L/K$ is cyclic Galois of degree $n$. Because $\characteristic(K)=0$ and $K$ contains all $n$-th roots of unity, Kummer theory provides a $u\in L^\times$ such that $L=K(u)$ and $u^n\in K^\times$.

We now claim that $\sigma^*(u)=\zeta u$ for some primitive $n$-th root of unity $\zeta\in \mathbb{C}$. Note that
\[
\sigma^*(u)^n = \sigma^*(u^n) = u^n,
\]
where the last equality follows from the fact that $u^n\in K$. Therefore, $\sigma^*(u)^n/u^n = 1$, or equivalently, $(\sigma^*(u)/u)^n=1$. Thus $\sigma^*(u)/u=\zeta$, where $\zeta$ is an $n$-th root of unity in $L$. Since the constant field of $L$ is $\mathbb{C}$, it follows that $\zeta\in\mathbb{C}$. If $\zeta$ had order $d < n$, then $(\sigma^*)^d$ would act trivially on $K(u)=L$, contradicting our hypothesis that $\sigma^*$ has order $n$. Therefore, $\zeta$ is a primitive $n$-th root of unity.

Now, choose any nonconstant $v\in K$ and define the rational map $\phi:Z\dashrightarrow\cp$ by $\mathbf{p}\mapsto [u(\mathbf{p}):v(\mathbf{p}):1]$. Let $U\subset Z$ be the set where $u$ and $v$ are regular. Then on $U\cap\sigma^{-1}(U)$
\[
(\phi\circ\sigma)(\mathbf{p}) = [(\sigma^*u)(\mathbf{p}):(\sigma^*v)(\mathbf{p}):1] = [\zeta u(\mathbf{p}):v(\mathbf{p}):1] = (\tau\circ\phi)(\mathbf{p}).
\]
Since $Z$ is a smooth projective curve and $\cp$ is proper, the rational map $\phi$ extends uniquely to a morphism, still denoted $\phi$, from $Z$ to $\cp$. Because $\phi\circ \sigma$ and $\tau\circ\phi$ agree on a dense open subset of $Z$, they are equal on all of $Z$. Finally, since $v$ was chosen to be nonconstant, $\phi$ is nonconstant.
\end{proof}

\begin{rem}
It is worth mentioning that the previous theorem also allows for the visualization of automorphisms under the more traditional projections $(u_0,u_1,v_0,v_1)\mapsto(u_0,u_1,v_0)$. However, these will not naturally extend to $\cp$ and therefore give a more limited view of the symmetries involved.
\end{rem}

To apply the above theorem, we must have a function field that is manageable to work with. In such cases, however, the theorem does provide a pathway to visualize the automorphism as a rotation of $\widetilde{\alpha}(\phi(Z))$.

\begin{exmp}
Let $Z=\mathbb{P}_\mathbb{C}^1$, $n>1$, and $\zeta=e^{2\pi i/n}$. Consider the automorphism of $Z$'s function field, $\mathbb{C}(z)$, given by
\[
\sigma^*(z)=\frac{\zeta z}{(\zeta-1)z+1}.
\]
A simple inductive argument yields
\[
\left(\sigma^*\right)^j(z)=\frac{\zeta^j z}{(\zeta^j-1)z+1}.
\]
Thus, $\sigma^*$ has order $n$.

In the spirit of the preceding theorem, we wish to construct a $u\in \mathbb{C}(z)$ such that $\sigma^*(u)=\zeta u$, and a generator $v$ of $\mathbb{C}(z)^{\sigma^*}$. To begin, we find $u$ by diagonalizing $\sigma^*$.
\begin{align*}
\begin{pmatrix}
\zeta & 0\\
\zeta-1 & 1
\end{pmatrix}&=
\begin{pmatrix}
1 & 0\\
1 & -1
\end{pmatrix}
\begin{pmatrix}
\zeta & 0\\
0 & 1
\end{pmatrix}
\begin{pmatrix}
1 & 0\\
1 & -1
\end{pmatrix}\\
\begin{pmatrix}
1 & 0\\
1 & -1
\end{pmatrix}
\begin{pmatrix}
\zeta & 0\\
\zeta-1 & 1
\end{pmatrix}&=
\begin{pmatrix}
\zeta & 0\\
0 & 1
\end{pmatrix}
\begin{pmatrix}
1 & 0\\
1 & -1
\end{pmatrix}.
\end{align*}
Let $u=z/(z-1)$. Then, by the above diagonalization,
\[
\sigma^*(u(z)) = u(\sigma^*(z))=\zeta u(z).
\]
Thus the field fixed by $\sigma^*$ is $\mathbb{C}(z)^{\sigma^*}=\mathbb{C}(u^n)$. We can now select $v$ to be any generator of $\mathbb{C}(u^n)$. For instance, take $v=u^n/(u^n-1)$. Then the rational map $\phi:Z\dashrightarrow\cp$, given by $\mathbf{p}\mapsto[u(\mathbf{p}):v(\mathbf{p}):1]$, extends to a morphism $\phi:Z\rightarrow \mathbf{V}(vu^n-vw^n-u^nw)\subset\cp$. Moreover, 
\[
(\phi\circ\sigma)(\mathbf{p})=[(\sigma^*u)(\mathbf{p}):(\sigma^*v)(\mathbf{p}):1] = [\zeta u(\mathbf{p}): v(\mathbf{p}):1].
\]
By Theorem \ref{thm:diagonal_automorphism}, $\sigma$ descends to a rotation of $\widetilde{\alpha}(\phi(Z))$ about the $x$-axis by $-2\pi/n$. To illustrate a particular instance, we let $n=3$.
\begin{figure}[!htb]
    \centering
    \includegraphics[height=1.5in]{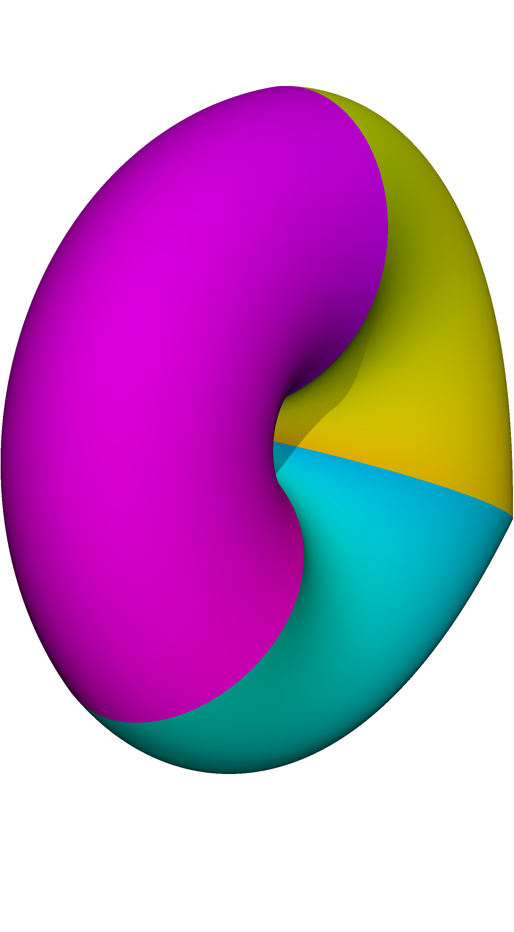}\hfill
    \includegraphics[height=1.5in]{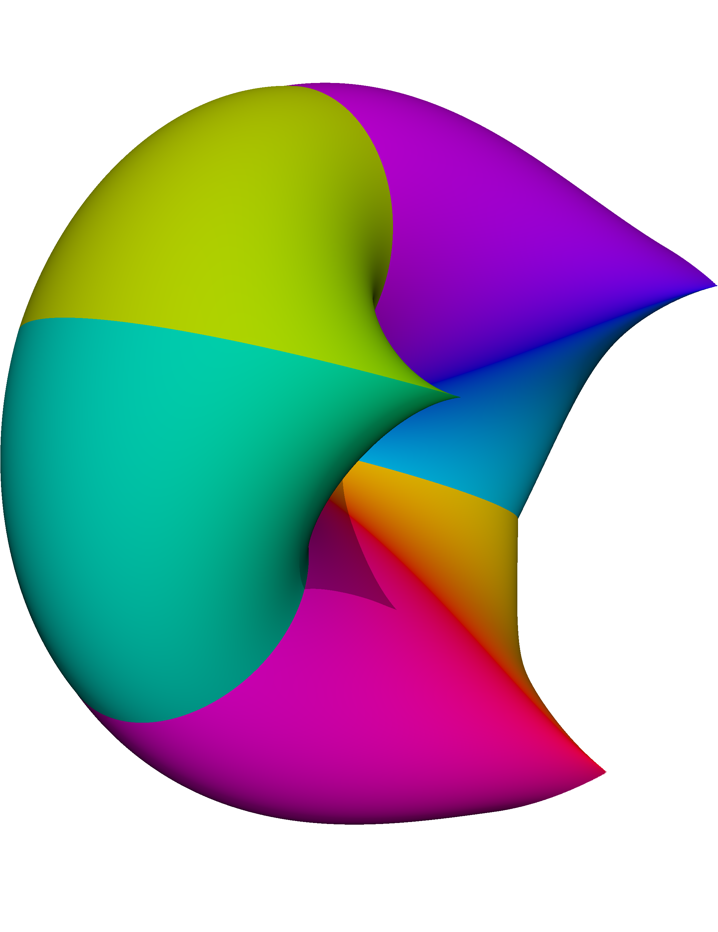}\hfill
    \includegraphics[height=1.5in]{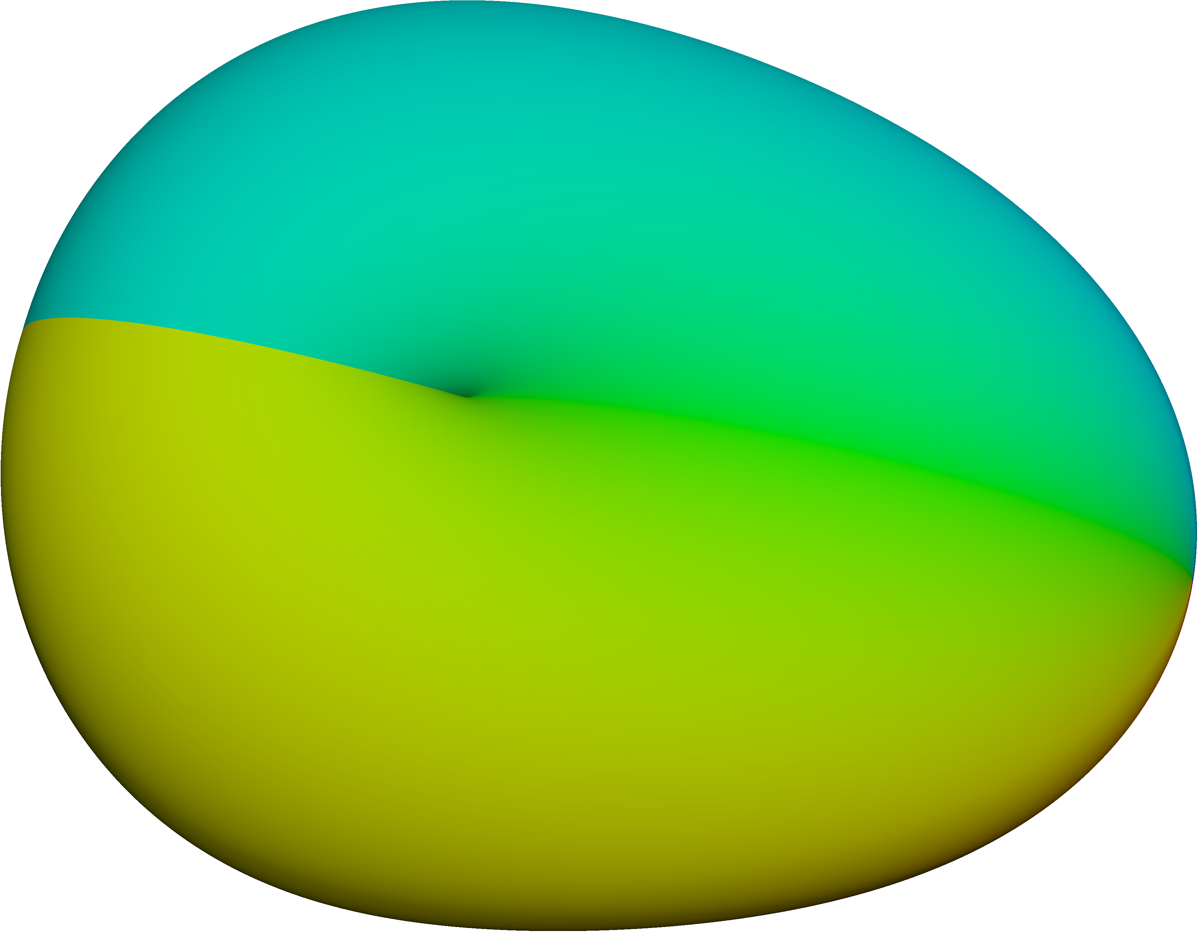}\hfill
    \includegraphics[height=1.5in]{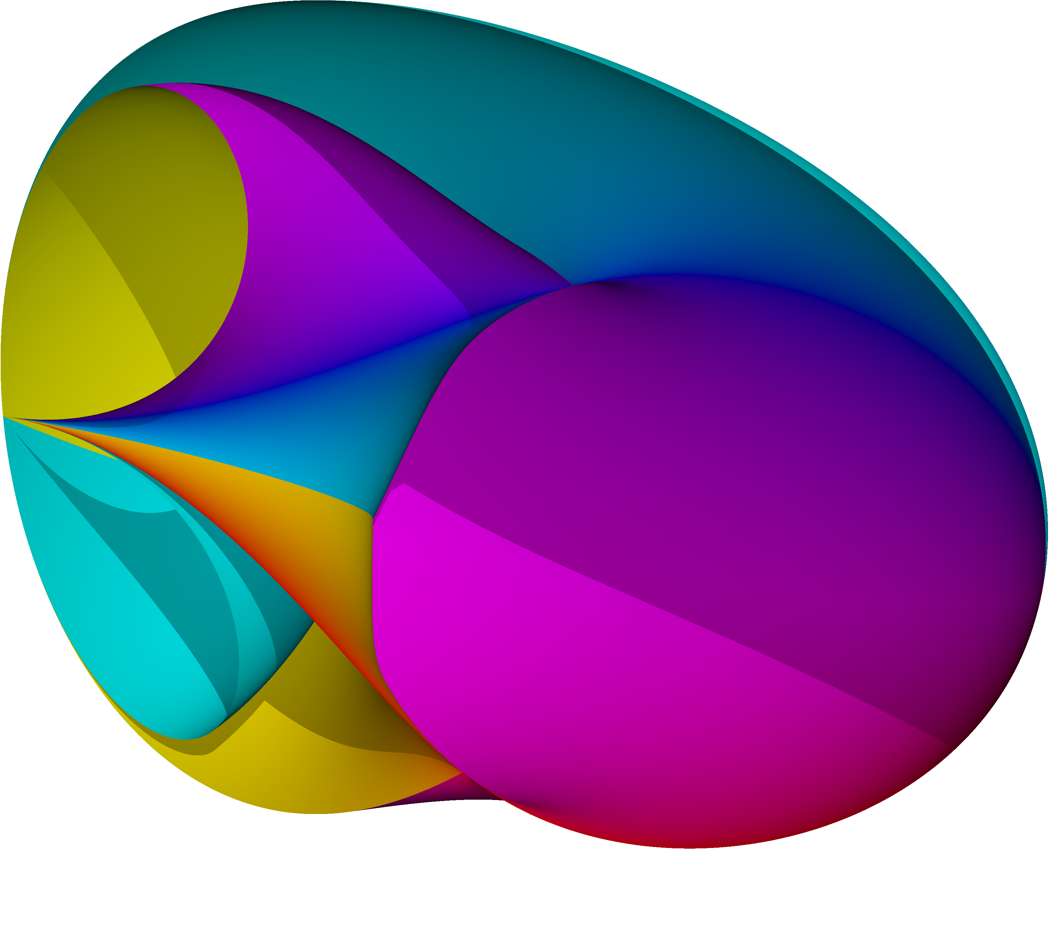}
    \caption{Ray-traced views of $\partial S_f(2)$, $\partial S_f(3)$, $\partial S_f(4)$, and a cutaway view of their union.}
\end{figure}

Note that we could just as well have chosen $v=u^n$. The above $v$ was selected for aesthetic reasons.
\end{exmp}

\subsection{Normal vectors}

At a smooth point of an implicitly defined surface in $\mathbb{R}^3$, the gradient of the defining equation provides a normal vector. For $\alpha(X)$, however, we generally do not have a convenient implicit equation. We therefore compute normals indirectly, using the complex tangent line to $X$. When $\alpha$ is an immersion and a local embedding near a point of $X$, the image of the tangent line is an ellipsoid that is tangent to $\alpha(X)$, so a normal vector to that ellipsoid is also normal to $\alpha(X)$. Theorem \ref{thm:normal_vector} provides a formula for computing these normal vectors to $\alpha(X)$ in terms of the partial derivatives of $f(u,v)$.

\begin{prop}\label{prop:line-ellipsoid}
Let $a,b,c\in\mathbb{C}$ with $c\neq 0$ and $(a,b)\neq (0,0)$. The image under $\alpha$ of the complex line
\[
au+bv=c
\]
restricted to $\A$ is the surface in $B$ defined by
\begin{equation}\label{eq:ellipse_eq}
|ax+b(y+iz)|^2=|c|^2(x-x^2-y^2-z^2),
\end{equation}
which, when expanded, gives the ellipsoid
\[
|a|^2x^2+|b|^2y^2+|b|^2z^2+2xy\Real(a\bar{b})+2xz\Imag(a\bar{b})=|c|^2(x-x^2-y^2-z^2).
\]
\end{prop}

\begin{proof}
Let $(u,v)\in \A$ lie on the given line and $\alpha(u,v)=(x,y,z)\in B$. Theorem \ref{thm:uv_formula} gives a $\lambda\in S^1$ such that
\[
(u,v)=\left(\frac{x\lambda}{\sqrt{x-x^2-y^2-z^2}},\frac{(y+iz)\lambda}{\sqrt{x-x^2-y^2-z^2}}\right).
\]
Substituting the above into $au+bv=c$ and clearing denominators yields
\[
\lambda\left(ax+b(y+iz)\right)=c\sqrt{x-x^2-y^2-z^2}.
\]
Taking the square of the absolute value of both sides and using the fact that $|\lambda|=1$ gives the desired equation.

Conversely, suppose $(x,y,z)\in B$ satisfies Equation \eqref{eq:ellipse_eq}. Since $c\neq 0$ and $x-x^2-y^2-z^2>0$, the quantity $ax+b(y+iz)$ is nonzero. Define
\[
\lambda=\frac{c\sqrt{x-x^2-y^2-z^2}}{ax+b(y+iz)}.
\]
Then by the assumption that $(x,y,z)$ satisfies Equation \eqref{eq:ellipse_eq}, $|\lambda|=1$. Setting $(u,v)=\beta(x,y,z,\lambda)$, by Remark \ref{rem:comp_ident} we have $\alpha(u,v)=(x,y,z)$ and
\begin{align*}
au+bv &= \frac{ax\lambda}{\sqrt{x-x^2-y^2-z^2}} + \frac{b(y+iz)\lambda}{\sqrt{x-x^2-y^2-z^2}}\\
&=\frac{\lambda(ax+b(y+iz))}{\sqrt{x-x^2-y^2-z^2}}\\
&=c.
\end{align*}
Thus $(x,y,z)$ lies in the image of the line.
\end{proof}

\begin{rem}
The quadric in Proposition \ref{prop:line-ellipsoid} is contained in $\overline{B}$. Indeed, rewriting its equation gives
\[
x=\left|\frac{ax+b(y+iz)}{c}\right|^2+x^2+y^2+z^2\ge x^2+y^2+z^2.
\]
\end{rem}

\begin{prop}\label{prop:jacobian}
Viewed as a real function, $\alpha:\mathbb{R}^4\rightarrow\mathbb{R}^3$ has a Jacobian of maximal rank at every point outside of $\{(0,0)\}\times \mathbb{R}^2$.
\end{prop}

\begin{proof}
Writing $\alpha$ as a real function gives
\[
\alpha(u_0, u_1, v_0, v_1) = \left(\frac{u_0^2+u_1^2}{1+u_0^2+u_1^2+v_0^2+v_1^2}, \frac{u_0v_0+u_1v_1}{1+u_0^2+u_1^2+v_0^2+v_1^2}, \frac{u_0v_1-u_1v_0}{1+u_0^2+u_1^2+v_0^2+v_1^2}\right).
\]
Let $|J_{\alpha, k}|$ denote the determinant of the minor of the Jacobian with the $k$-th column removed. Then
\begin{equation}\label{eq:jacobians}
\big(-|J_{\alpha, 1}|, |J_{\alpha, 2}|, -|J_{\alpha, 3}|, |J_{\alpha, 4}|\big) = (-u_1, u_0, -v_1, v_0)\frac{2(u_0^2+u_1^2)}{(1+u_0^2+u_1^2+v_0^2+v_1^2)^4},
\end{equation}
where the calculations are left to the reader. However, the right-hand side is the zero vector if and only if $u_0 = u_1 = 0$. If it is not the zero vector, then the determinant of one of the $3\times 3$ minors is nonzero, and the rank of the Jacobian must be $3$. Therefore, the Jacobian has maximal rank outside of the set $u_0=u_1=0$.
\end{proof}

We recall the following fact from linear algebra.

\begin{lemma}\label{lem:kernel}
Let $A$ be an $n\times (n+1)$ real matrix and $|A_k|$ denote the determinant of the minor obtained by removing the $k$-th column. Then the vector
\[
\vec{c} = \big((-1)^{n+2}|A_1|, (-1)^{n+3}|A_2|,\ldots, (-1)^{2n+2}|A_{n+1}|\big)
\]
is in the kernel of $A$.
\end{lemma}

\begin{proof}
It suffices to show that the rows of $A$ are all orthogonal to $\vec{c}$. Let $(a_{j1}, a_{j2},\ldots,a_{j(n+1)})$ be the $j$-th row of $A$. Then
\begin{align*}
(a_{j1}, a_{j2},\ldots,a_{j(n+1)})\cdot \vec{c} &= \sum_{k=1}^{n+1} (-1)^{n+1+k}a_{jk}|A_k|\\
&=\begin{vmatrix}
a_{11} & a_{12} & \cdots & a_{1(n+1)}\\
\vdots & \vdots & \ddots & \vdots\\
a_{n1} & a_{n2} & \cdots & a_{n(n+1)}\\
a_{j1} & a_{j2} & \cdots & a_{j(n+1)}
\end{vmatrix}.
\end{align*}
However, this determinant is $0$ as the $j$-th row of $A$ appears twice. Therefore, $\vec{c}$ is orthogonal to every row of $A$ and thus lies in the kernel.
\end{proof}

\begin{thm}
Let $S_1$ be the set of singular points of $X\subset \mathbb{C}^2$ and $S_2\subset X$ be the set of smooth points where the complex tangent line to $X$ contains the origin. Then $\alpha : (X\cap\A)\setminus (S_1\cup S_2) \rightarrow \mathbb{R}^3$ is an immersion of real manifolds.
\end{thm}

\begin{proof}
Viewing $U = (X\cap\A)\setminus (S_1\cup S_2)$ as a real manifold, fix $\mathbf{p} = (u_0, u_1, v_0, v_1)\in U$. Since $u\neq 0$ on $U$, the Jacobian of $\alpha$ has maximal rank at $\mathbf{p}$ by Proposition \ref{prop:jacobian}. Therefore, by Lemma \ref{lem:kernel} and Equation \eqref{eq:jacobians}, a vector is in the kernel if and only if it is a real multiple of $(-u_1, u_0, -v_1, v_0)$. Suppose that the differential of $\alpha|_U$ is not injective at $\mathbf{p}$. Then the real line
\[
(u_0, u_1, v_0, v_1) + t(-u_1, u_0, -v_1, v_0)
\]
must be contained within the tangent plane to $U$ at $\mathbf{p}$. Viewed as a subset of $\mathbb{C}^2$, this is the real line $\ell = (u_0+iu_1, v_0+iv_1)(1+it)$, which is contained within the complex line $(u_0+iu_1, v_0+iv_1)z$, where $z$ runs over all of $\mathbb{C}$. If $\ell$ were contained in another complex line, then we would have two distinct complex lines with infinitely many points in common, which is a contradiction. Therefore, this is the unique complex line containing $\ell$. In particular, it must be the complex tangent line to $U$ at $\mathbf{p}$. Letting $z=0$ we see that the tangent line contains the origin, which contradicts our construction of $U$. Therefore, the differential of $\alpha|_U$ is injective, and $\alpha|_U:U\rightarrow \mathbb{R}^3$ is an immersion.
\end{proof}

Even when $\alpha$ is an immersion, it need not be globally injective. Distinct local branches of $X$ may meet in the image and create pinch-type singularities.

\begin{thm}\label{thm:normal_vector}
Let $\mathbf{p}=(u,v)\in X\cap\A$ be a smooth point, and suppose there is an open neighborhood $U\subset X$ of $\mathbf{p}$ such that $\alpha|_U$ is an embedding. Write
\[
\alpha(\mathbf{p})=(x,y,z),
\]
and let $f_u$ and $f_v$ denote the complex partial derivatives of the defining polynomial $f$. Then a normal vector to $\alpha(U)$ at $\alpha(\mathbf{p})$ is
\begin{align*}
\bigg(&2x|f_u(\mathbf{p})|^2+2y\Real \left(f_u(\mathbf{p})\overline{f_v(\mathbf{p})}\right) + 2z\Imag \left(f_u(\mathbf{p})\overline{f_v(\mathbf{p})}\right)+(2x-1)|uf_u(\mathbf{p})+vf_v(\mathbf{p})|^2,\\
&2y|f_v(\mathbf{p})|^2+2x\Real \left(f_u(\mathbf{p})\overline{f_v(\mathbf{p})}\right) + 2y|uf_u(\mathbf{p})+vf_v(\mathbf{p})|^2,\\
&2z|f_v(\mathbf{p})|^2+2x\Imag \left(f_u(\mathbf{p})\overline{f_v(\mathbf{p})}\right) + 2z|uf_u(\mathbf{p})+vf_v(\mathbf{p})|^2\bigg).
\end{align*}
\end{thm}

\begin{proof}
The formula comes from applying Proposition \ref{prop:line-ellipsoid} to the complex tangent line
\[
f_u(\mathbf{p})(u'-u)+f_v(\mathbf{p})(v'-v)=0
\]
of $X$ at $\mathbf{p}$, and then taking the gradient of the resulting ellipsoid at $\alpha(\mathbf{p})$.
\end{proof}

\section{Visualization}\label{section:visualization}

Rendering images of $\alpha(X)$ appears to encounter many of the same difficulties as other visualization techniques. One could use branch cuts to construct a mesh of $X$ and then apply $\alpha$. Alternatively, if $X$ is rational or admits a convenient parameterization, then creating a mesh and mapping it to $\mathbb{R}^3$ with $\alpha$ is efficient.

In this section we provide two additional methods for rendering $\alpha(X)$. The first, in Section \ref{subsection:mesh_contraction}, contracts meshes of $\partial B$ onto the boundaries $\partial S_f(k)$. This is analogous to using branch cuts in the more traditional projections of Riemann surfaces. The second, in Section \ref{subsection:ray_tracing}, ray traces the image directly in $\mathbb{R}^3$.

The two methods require numerical root finding and root counting, respectively. The particular root-finding and root-counting methods suggested here are not the only possible choices, but in the author's experience they accurately resolve the surface near self-intersections and in regions of fine detail. Our goal is a \emph{visually faithful} depiction of $\alpha(X)$, so at times an arbitrary choice of branch may be required for coloring or shading. Moreover, as with any method based on sampling along a ray, sufficiently narrow features may be missed.

We keep the conventions from the previous sections. Namely, $X\subset\mathbb{C}^2$ is irreducible, not a line through the origin, and defined by a polynomial $f(u,v)$ of degree $n$. We write $Z\subset\cp$ for the projective closure of $X$.

\subsection{Mesh contraction}\label{subsection:mesh_contraction}

Our goal is to render $\widetilde{\alpha}(Z)$. By Lemma \ref{lem:boundary_points} and Theorem \ref{thm:projective_image}, it suffices to construct meshes for the boundaries $\partial S_f(k)$. Note that some of the sets $S_f(k)$ may be empty. Indeed, if $(0,0)\in X$ has multiplicity $m$, then $u=0$ is a root of $f(u,cu)$ of multiplicity at least $m$ for every $c\in\mathbb{C}$, so $N_f(\mathbf{p})\geq m$ for every $\mathbf{p}\in B$. Hence $S_f(k)=\emptyset$ for $1\leq k\leq m$.

For each nonempty $S_f(k)$, a mesh approximating $\partial S_f(k)$ can be constructed by contracting a triangular mesh of $\partial B$ along rays from the origin. To do so, begin by constructing a mesh $\mathcal{T}$ of $\partial B$, containing $(0,0,0)$ as a vertex. Fix any other vertex $\mathbf{p}=(x,y,z)$. Since $\mathbf{p}\in\partial B\setminus\{(0,0,0)\}$, we have $x>0$. Determine the roots of the polynomial
\[
u\mapsto f\left(u,\frac{y+iz}{x}u\right)
\]
and add roots at infinity so that there are $n$ roots counted with multiplicity. Order them by their moduli $|u_1|\leq |u_2|\leq\cdots\leq |u_n|$, where $|\infty|=\infty$. In the author's experience, either the Aberth--Ehrlich method or finding the eigenvalues of the companion matrix was efficient and numerically stable for computing the finite roots.

For finite $u_k$, a straightforward calculation, along with the fact that $x^2+y^2+z^2=x$ on $\partial B$, gives
\[
\alpha\left(u_k,\frac{y+iz}{x}u_k\right) = \frac{|u_k|^2}{x+|u_k|^2}(x, y, z).
\]
For convenience, define $\rho_k = |u_k|^2/(x+|u_k|^2)$ when $u_k$ is finite and $1$ otherwise. Replacing each vertex $\mathbf{p}$ by $\rho_k\mathbf{p}$ gives a mesh approximation of $\partial S_f(k)$.

After moving a vertex, assign a hue using $\arg(u_k)$. If distinct finite roots have the same modulus, the choice of $u_k$, and hence its argument and hue, is not well defined. For finite roots with equal modulus, using the root returned in the $k$-th position consistently produced good results in the author's experience. Additionally, the hue is undefined when $u_k=0$ or $u_k=\infty$, in which case one may omit the saturation. Since all roots are simultaneously computed at each vertex, all of the shell meshes can be constructed at once.

\begin{algorithm}[!ht]
\caption{Mesh contraction}
\begin{algorithmic}[1]
\Function{ContractMeshes}{$f,\mathcal{T},K=\{1,\ldots,n\}$}
    \ForAll{$k\in K$}
        \State $\mathcal{T}_k\gets\mathcal{T}$
    \EndFor
    \ForAll{$\mathbf{p}=(x,y,z)\in\operatorname{Vert}(\mathcal{T})\setminus\{\mathbf{0}\}$}
        \State
        $(u_1,\ldots,u_n)\gets\Call{ProjectiveRoots}{f\left(u,\frac{y+iz}{x}u\right)}$
        \State Sort $u_1,\ldots,u_n$ by modulus
        \ForAll{$k\in K$}
            \State
            $\displaystyle \rho_k\gets
            \begin{cases}
                |u_k|^2/(x+|u_k|^2), & u_k\in\mathbb{C},\\
                1, & u_k=\infty
            \end{cases}$
            \State Move the corresponding vertex of
            $\mathcal{T}_k$ to $\rho_k\mathbf{p}$
            \State Set hue to $\arg(u_k)$, when defined
        \EndFor
    \EndFor
    \State \Return $\{\mathcal{T}_k\}_{k\in K}$
\EndFunction
\end{algorithmic}
\end{algorithm}
\FloatBarrier

The following figures were rendered in MeshLab after being constructed using the above algorithm.

\begin{figure}[!ht]
    \centering
    \begin{subfigure}{0.4\textwidth}
        \centering
        \includegraphics[height=1.5in]{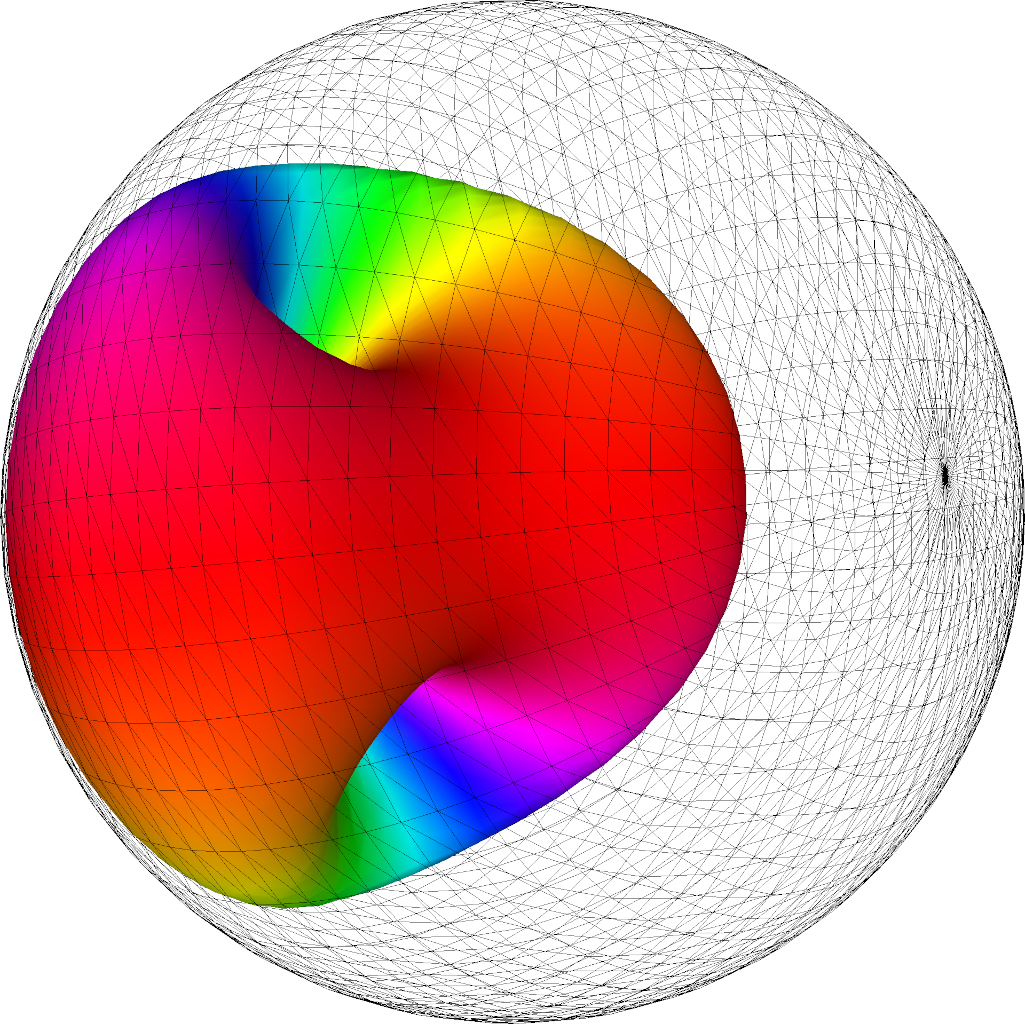}
        \caption{Initial spherical mesh and $\partial S_f(3)$}
    \end{subfigure}
    \hfill
    \begin{subfigure}{0.4\textwidth}
        \centering
        \includegraphics[height=1.5in]{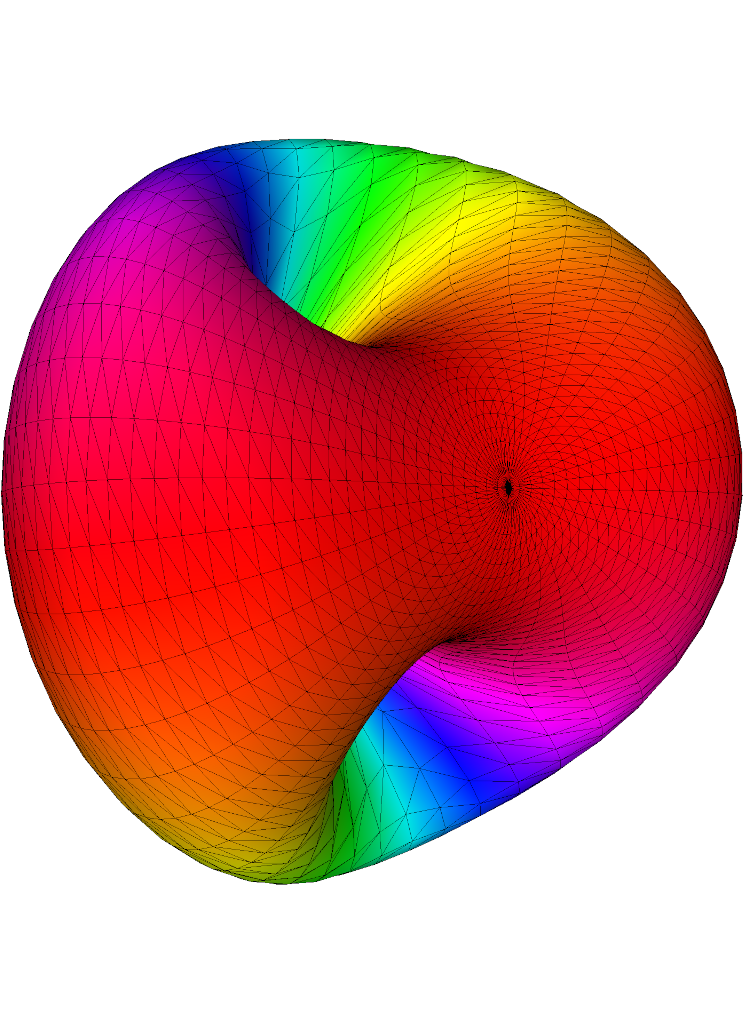}
        \caption{Mesh contracted to $\partial S_f(3)$}
    \end{subfigure}
    \caption{The nodal cubic given by $f(u,v)=v^2-u^2-u^3$. Since $X$ has multiplicity two at $(0,0)$, the sets $S_f(1)$ and $S_f(2)$ are empty. Thus $\widetilde{\alpha}(Z)=\partial S_f(3)$, and only one spherical mesh is contracted.}
\end{figure}
\FloatBarrier

\begin{figure}[!htb]
    \centering
    \begin{subfigure}{0.3\textwidth}
        \centering
        \includegraphics[height=1.5in]{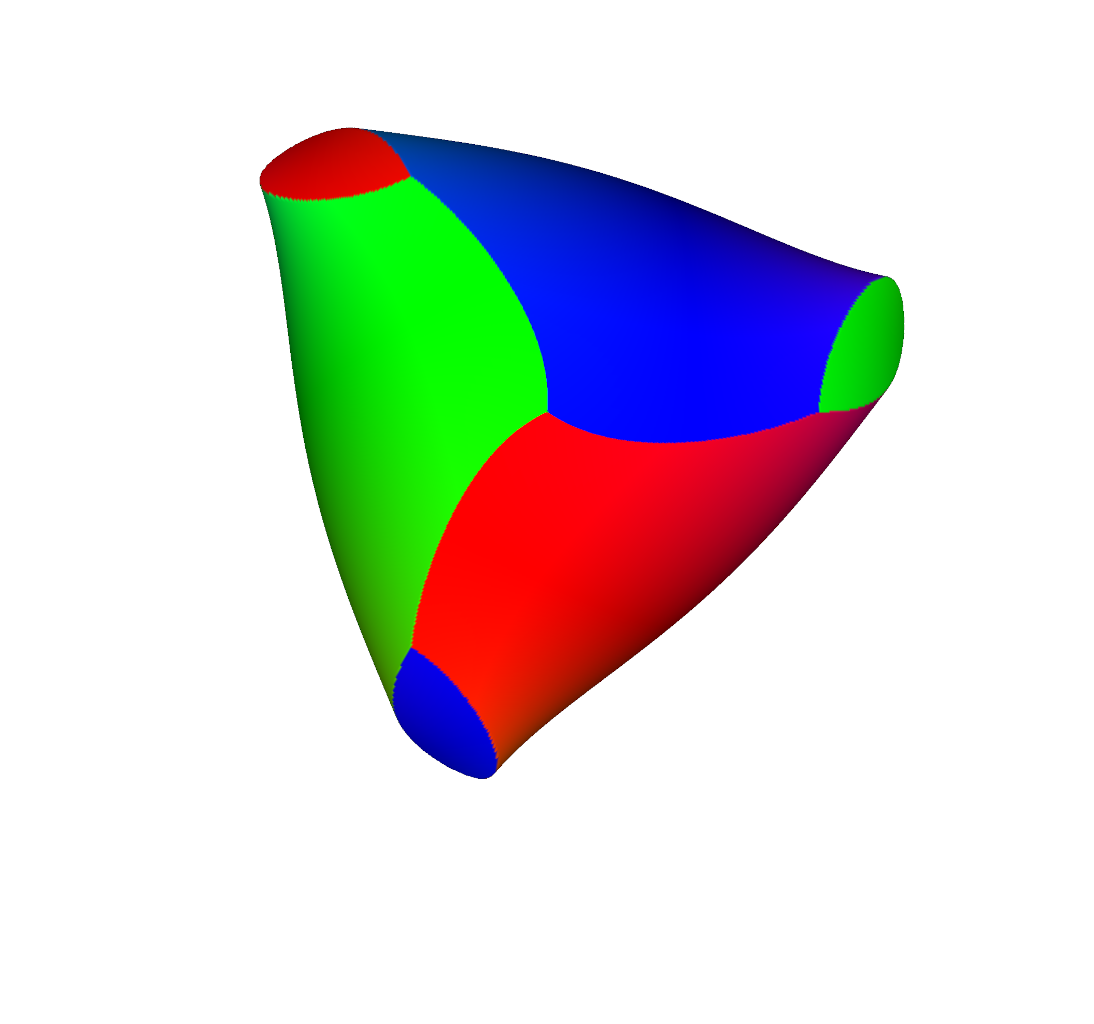}
        \caption{$\partial S_f(1)$}
    \end{subfigure}
    \hfill
    \begin{subfigure}{0.3\textwidth}
        \centering
        \includegraphics[height=1.5in]{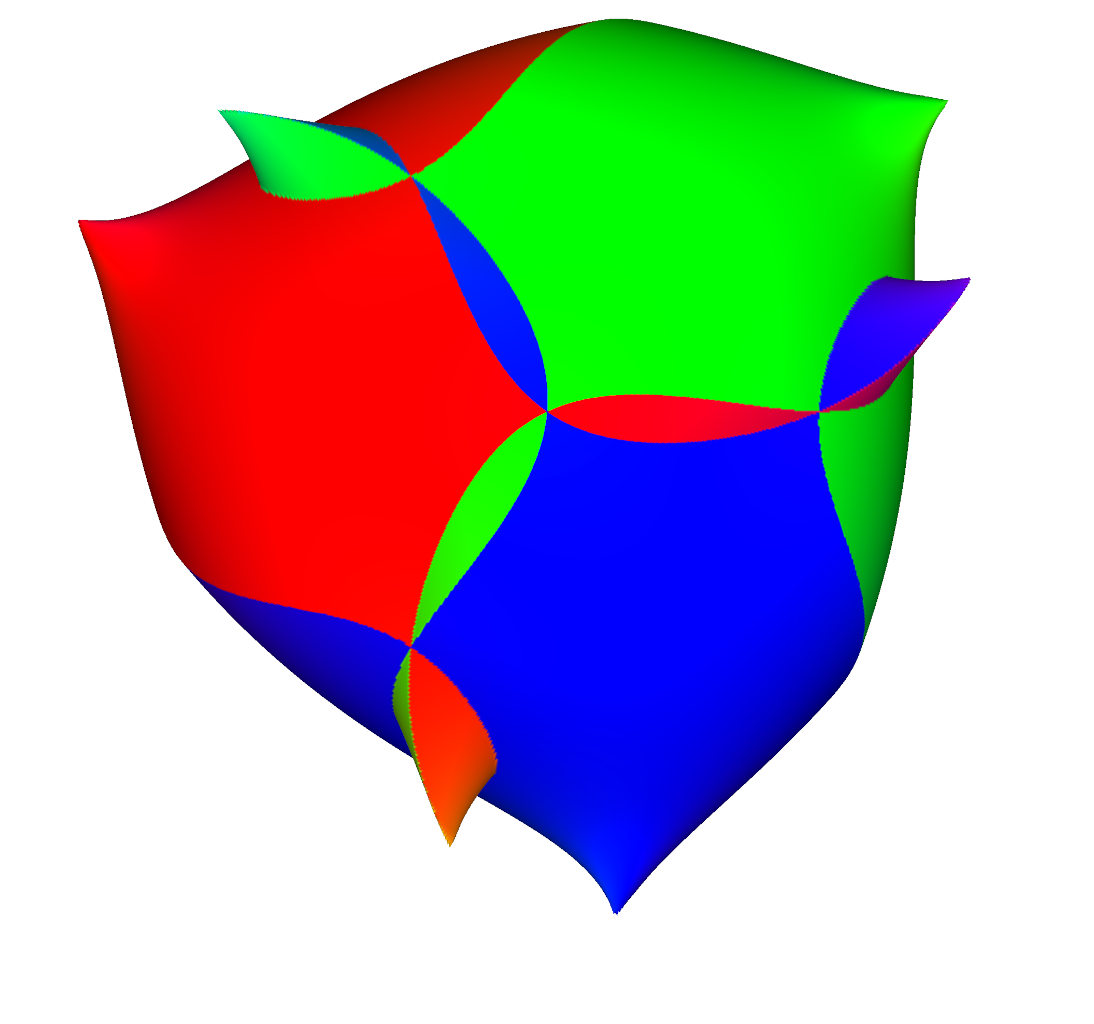}
        \caption{$\partial S_f(2)$}
    \end{subfigure}
    \hfill
    \begin{subfigure}{0.3\textwidth}
        \centering
        \includegraphics[height=1.5in]{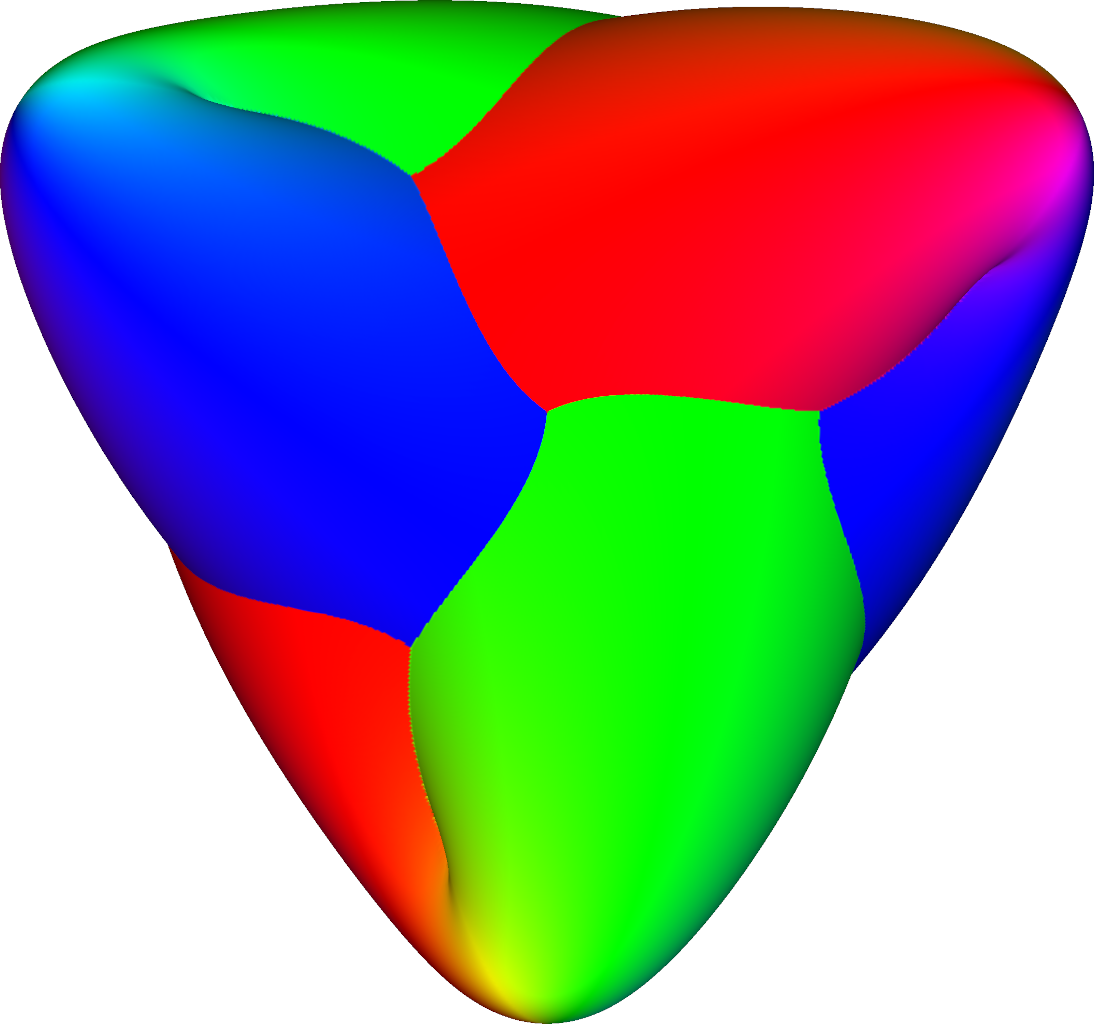}
        \caption{$\partial S_f(3)$}
    \end{subfigure}
    \caption{The curve $f(u,v)=(v+i/2)^3+u^3+1$ is a translation of the affine Fermat cubic. Since $(0,0)\notin X$, all three sets, $S_f(1)$, $S_f(2)$, and $S_f(3)$, are nonempty.}
\end{figure}
\FloatBarrier

\subsection{Ray tracing}\label{subsection:ray_tracing}

For brevity, we do not cover the theory of ray tracing in its entirety. Instead, we focus on how to compute ray-surface intersections and find the corresponding normal vectors on $\widetilde{\alpha}(Z)$. These techniques are sufficient to adapt standard implicit-surface ray tracing to create visualizations of compact Riemann surfaces. As in the previous section, the particular implementation will focus on $\alpha(X)$ as this contains all but finitely many points of $\widetilde{\alpha}(Z)$.

Let $\mathbf{o}$ be the origin of the ray, typically the camera or eye, and $\mathbf{d}$ the direction vector. We then march through the portion of the ray $\mathbf{o}+t\mathbf{d}$ contained in $B$. At each sampled point $\mathbf{p}_i$ along the ray, we compute $N_f(\mathbf{p}_i)$. Recall that this is the number of zeros, counted with multiplicity, of the polynomial
\[
\lambda\mapsto f(\beta(\mathbf{p}_i,\lambda))
\]
in the open unit disk $|\lambda|<1$. The roots can be counted using any root-finding or root-counting method. In the author's experience, the Schur--Cohn algorithm has proven highly effective, since it is tailored to counting roots in the unit disk and straightforward to implement in parallel on a GPU.

Suppose the root counts at two consecutive samples, $N_f(\mathbf{p}_i)$ and $N_f(\mathbf{p}_{i+1})$, differ. By Theorem \ref{thm:line_segment}, the segment between them intersects $\alpha(X)$ (see Figure \ref{fig:ray_tracing} for a visualization of this process). We then apply the bisection method until the interval is within the desired tolerance, which should be chosen to correspond to at least subpixel precision. The midpoint of this interval gives an approximation of the intersection. Numerically, tangential intersections are unlikely and need not be considered. This method is also far more robust at detecting multiple crossings than traditional ray tracing due to the fact that we are counting the number of zeros rather than simply checking for sign changes. Even so, there is always the possibility that the intersection with the surface is not detected by this process. For higher-quality images it is beneficial to use sufficiently small steps in the ray marching process.

\begin{figure}[!htb]
    \centering
    \begin{subfigure}[b]{0.56\textwidth}
        \centering
        \includegraphics[height=1.9in]{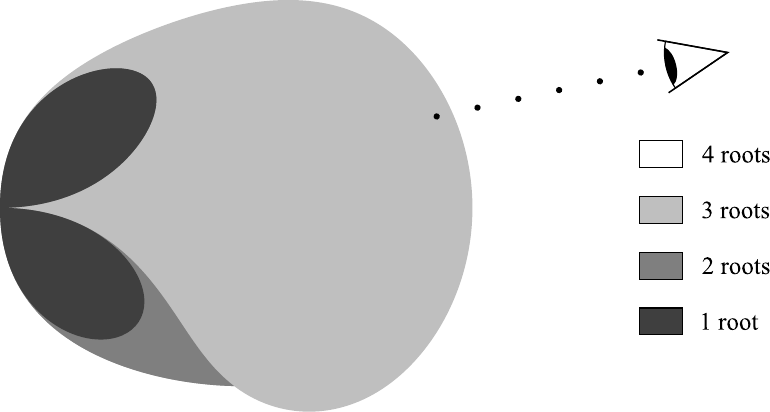}
        \caption{Root-counting in a slice of $B$}
    \end{subfigure}
    \hfill
    \begin{subfigure}[b]{0.38\textwidth}
        \centering
        \includegraphics[height=1.9in]{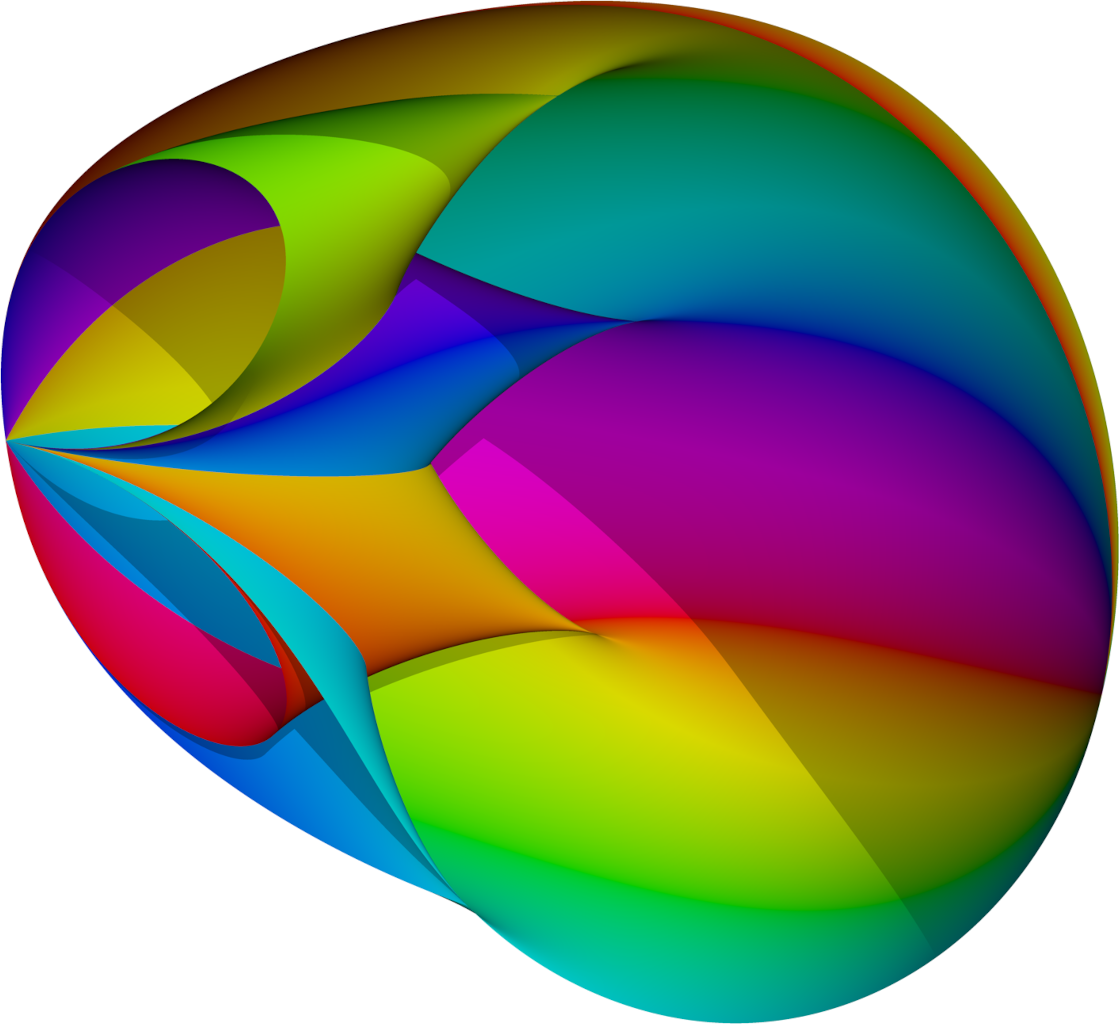}
        \caption{Cutaway view of the ray-traced surface}
    \end{subfigure}
    \caption{A quartic ray traced by detecting changes in $N_f$.}
    \label{fig:ray_tracing}
\end{figure}
\FloatBarrier

Unfortunately, finding a point in the image of $X$ is not sufficient for ray tracing the surface at that point. In order to compute lighting and reflections, a normal vector is needed. Additionally, to assign a hue, we need to know $\arg(u)$. Both of these issues can be resolved by finding the original point on $X$.

Suppose that we have a point $\mathbf{p}$ that is established to be approximately in $\alpha(X)$. Consider its preimage, which, by Theorem \ref{thm:uv_formula}, is the circle $\left\{\beta(\mathbf{p},\lambda):|\lambda|=1\right\}$. If we were working exactly, then the points on $X$ mapping to $\mathbf{p}$ would be those points $\beta(\mathbf{p},\lambda)$ for which $f(\beta(\mathbf{p},\lambda))=0$ and $|\lambda|=1$. Since we are not working exactly, we let $\lambda$ be a root of $f(\beta(\mathbf{p},\lambda))$ with modulus closest to $1$. Then $\beta(\mathbf{p},\lambda)$ is taken to be the point on $X$ corresponding to $\mathbf{p}$. Applying Theorem \ref{thm:normal_vector} to $X$ at $\beta(\mathbf{p},\lambda)$ gives an approximation of the normal vector to $\alpha(X)$ at $\mathbf{p}$. Although Theorem \ref{thm:normal_vector} requires that $\alpha$ is locally an embedding, the set of points where this fails has measure zero and, in the author's experience, ignoring this requirement has not created any visibly noticeable aberrations when ray traced at high resolution. Finally, we use this same point on $X$ to assign the hue.

\begin{algorithm}[!ht]
\caption{Ray tracing}
\begin{algorithmic}[1]
\Function{TraceRay}{$f,r,\Delta t,\varepsilon$}
    \State Sample $r\cap B$ with step size $\Delta t$
    \State Let $\mathbf{a},\mathbf{b}$ be the first adjacent samples such that
    $N_f(\mathbf{a})\neq N_f(\mathbf{b})$
    \If{no such pair exists}
        \State \Return \textsc{Miss}
    \EndIf
    \While{$\|\mathbf{a}-\mathbf{b}\|>\varepsilon$}
        \State $\mathbf{c}\gets(\mathbf{a}+\mathbf{b})/2$
        \If{$N_f(\mathbf{a})\neq N_f(\mathbf{c})$}
            \State $\mathbf{b}\gets\mathbf{c}$
        \Else
            \State $\mathbf{a}\gets\mathbf{c}$
        \EndIf
    \EndWhile
    \State $\mathbf{p}\gets(\mathbf{a}+\mathbf{b})/2$
    \State $\Lambda\gets\Call{Roots}{f(\beta(\mathbf{p},\lambda))}$
    \State Choose $\lambda_0\in\Lambda$ minimizing $\left||\lambda|-1\right|$
    \State $(u_0,v_0)\gets\beta(\mathbf{p},\lambda_0)$
    \State $\mathbf{n}\gets\Call{NormalVector}{f,u_0,v_0}$
    \State $\theta\gets\arg(u_0)$
    \State \Return $(\mathbf{p},\mathbf{n},\theta)$
\EndFunction
\end{algorithmic}
\end{algorithm}
\FloatBarrier

As with the previous mesh contraction technique, the author has had good results using the Aberth--Ehrlich method. In particular, by starting with initial guesses equally spaced on the unit circle, convergence to the root closest to the unit circle tends to be rapid. However, other root-finding techniques that find all roots could be substituted.

In this outline of the process we have glossed over the possibility of multiple roots on the unit circle. In practice, this is unlikely to happen due to floating-point approximations. In general, it suffices to ignore such matters. However, in the case that $\alpha$ creates a multiple cover of its image, there will always be multiple roots $\lambda$ whose moduli are indistinguishably close to $1$. When this happens, it is best to work in grayscale. Otherwise the resulting image will be dotted with colors from what is effectively a random choice of branch at each point. This is precisely why Figure \ref{fig:six_to_one} is rendered without color.

\section{Acknowledgements}
The author would like to thank Brenna Pasch for her work on a software implementation for visualizing a predecessor of the map $\alpha$ used in this paper.
\bibliographystyle{plain}
\bibliography{DutterProjection}

\end{document}